\documentclass[11pt,reqno,a4paper]{amsart}

\allowdisplaybreaks[3]

\usepackage[latin1]{inputenc}
\usepackage[english]{babel}
\usepackage[T1]{fontenc}

\usepackage{amsmath}
\usepackage{amsthm}
\usepackage{amsfonts}
\usepackage{amssymb}
\usepackage{amscd}
\usepackage{eucal}
\usepackage{upref}
\usepackage{fancyhdr}
\usepackage[all]{xy}
\usepackage{enumerate}

\usepackage{tikz}

\usepackage{hyperref}

\renewcommand{\mathbf}{\boldsymbol}

\theoremstyle{plain}
\newtheorem{theorem}{Theorem}[section]
\newtheorem{lemma}[theorem]{Lemma}
\newtheorem{proposition}[theorem]{Proposition}
\newtheorem{corollary}[theorem]{Corollary}

\newtheorem{claim}[theorem]{Claim}

\numberwithin{equation}{section}

\theoremstyle{definition}
\newtheorem{definition}[theorem]{Definition}
\theoremstyle{remark}
\newtheorem{remark}[theorem]{Remark}
\newtheorem{remarks}[theorem]{Remarks}
\newtheorem{example}[theorem]{Example}

\renewcommand{\theenumi}{\roman{enumi}}

\DeclareMathOperator{\identity}{id}
\DeclareMathOperator{\ad}{ad}
\DeclareMathOperator{\Ric}{Ric}
\DeclareMathOperator{\Vol}{Vol}
\DeclareMathOperator{\Rm}{Rm}
\DeclareMathOperator{\Ad}{Ad}

\let\GAMMA=\Gamma
\renewcommand{\Gamma}{\mathit{\GAMMA}}
\let\DELTA=\Delta
\renewcommand\Delta{\mathit{\DELTA}}
\let\THETA=\Theta
\renewcommand{\Theta}{\mathit{\THETA}}
\let\LAMBDA=\Lambda
\renewcommand{\Lambda}{\mathit{\LAMBDA}}
\let\XI=\Xi
\renewcommand{\Xi}{\mathit{\XI}}
\let\PI=\Pi
\renewcommand{\Pi}{\mathit{\PI}}
\let\SIGMA=\Sigma
\renewcommand{\Sigma}{\mathit{\SIGMA}}
\let\PHI=\Phi
\renewcommand{\Phi}{\mathit{\PHI}}
\let\PSI=\Psi
\renewcommand{\Psi}{\mathit{\PSI}}
\let\OMEGA=\Omega
\renewcommand{\Omega}{\mathit{\OMEGA}}

\title{Ricci Flow on homogeneous spaces with two isotropy summands}

\author{Maria Buzano}

\address{Department of Mathematics and Statistics, McMaster University, 1280 Main street West, Hamilton, Ontario L8S 4K1, Canada}
\email{\href{mailto:mbuzano@math.mcmaster.ca}{mbuzano@math.mcmaster.ca}}

\date{}

\begin{document}

\maketitle

\begin{abstract}
We consider the Ricci flow equation for invariant metrics on compact and connected homogeneous spaces whose isotropy representation decomposes into two irreducible inequivalent summands.
By studying the corresponding dynamical system, we completely describe the behaviour of the homogeneous Ricci flow on this kind of spaces.
Moreover, we investigate the existence of ancient solutions and relate this to the existence and non-existence of invariant Einstein metrics.
\end{abstract}

\section{Introduction}
Let $(M,g)$ be a Riemannian manifold of dimension $n$. A one-parameter family of Riemannian metrics $\{g(t)\}_{t\in[0,T)}$ on $M$ is said to be a \emph{Ricci flow} with initial metric $g$ if it satisfies
\begin{equation}\label{eqn:1}
\frac{\partial}{\partial t}(g(t))=-2\Ric(g(t)),
\end{equation}
with $g(0)=g$. 
The Ricci flow was first introduced by Hamilton in \cite{Hamilton:2} who showed that compact 3-manifolds with strictly positive Ricci curvature are space forms with positive curvature.

An important property of the Ricci flow is that it preserves symmetries of the initial metric $g$. 
This is due to the fact that the Ricci tensor is invariant under diffeomorphisms of the manifold $M$. 
It is then natural to investigate the Ricci flow equation on Riemannian manifolds which admit a transitive action by a closed Lie group $G$ of isometries. 
These are called homogeneous Riemannian manifolds.
On this kind of spaces, we can restrict our attention to those Riemannian metrics which are invariant under the action of $G$ and this property is preserved under the Ricci flow. 

In this paper, we are going to consider compact and connected homogeneous spaces such that the isotropy representation decomposes into two irreducible inequivalent summands.
We will completely describe the behaviour of the \emph{homogeneous Ricci flow} (HRF) on this kind of spaces. 
More precisely, we will prove that the HRF always develops a type I singularity in finite time and analyse the different singular behaviours that can occur, as we approach the singular time.
We will also investigate the existence of ancient solutions to the HRF and show how this topic relates to the existence and non-existence of invariant Einstein metrics.
In a subsequent paper, we aim to generalise some of these results to a higher number of summands.

Some of the homogeneous spaces that we consider are interesting because they include examples of compact homogeneous spaces which do not admit any invariant Einstein metric. 
In this case, the HRF cannot converge, as Ricci flows can only converge to Einstein metrics, so it is interesting to see what behaviours might occur in the flow.
The existence and non-existence of invariant Einstein metrics on compact homogeneous spaces have been studied extensively by many authors \cite{Wang-Ziller:1, Bohm:2, Bohm:1, Bohm-Wang-Ziller:1, Bohm-Kerr:1, Dickinson-Kerr:1}. 
The lowest dimensional non-existence example is the 12-dimensional manifold $SU(4)/SU(2)$ and it was found by Wang and Ziller \cite{Wang-Ziller:1}.
This manifold is also an example of a compact homogeneous space with two irreducible inequivalent summands in the isotropy representation.
Later, B\"ohm and Kerr \cite{Bohm-Kerr:1} showed that this is the least dimensional example of a compact homogeneous space which does not carry any invariant Einstein metric. 
New non-existence examples were produced by B\"ohm in \cite{Bohm:1}. 

We would also like to mention that HRF for other spaces has been studied before \cite{Anastassiou-Chrysikos:1,Isenberg-Jackson:1,Isenberg-Jackson-Lu:1,Lauret:2,Payne:1,Lauret:3}.

The outline of the paper is as follows.
In section \ref{sec:2}, we recall some basic notions and facts about homogeneous Riemannian manifolds, the Ricci flow equation and the convergence of metric spaces.
In section \ref{sec:3}, we consider the HRF on isotropy irreducible spaces.
In sections \ref{sec:5} and \ref{sec:1}, we analyse the HRF on compact and connected homogeneous spaces whose isotropy representation decomposes into two irreducible inequivalent summands, distinguishing the cases in which there exists an intermediate Lie algebra or not.
In section \ref{sec:4}, we analyse the Ricci soliton that we obtain, by blowing up the solution near the singular time. 
Finally, in section \ref{sec:7}, we make some remarks on the case in which the initial metric is pseudo-Riemannian.

\section*{Acknowledgements}
This paper is part of my PhD thesis \cite{Buzano:2}.
I would like to thank my supervisor Prof. Andrew Dancer for his current support and motivation and Chris Hopper, Prof. Ernesto Buzano and Prof. McKenzie Wang for useful discussions.
I would also like to thank Prof. Lei Ni for bringing to my attention reference \cite{Bakas-Kong-Ni:1}.
Finally, I would like to acknowledge the EPSRC and the Accademia delle Scienze di Torino for financial support.

\section{Preliminaries}\label{sec:2}

\subsection{The Ricci tensor of a homogeneous Riemannian manifolds}
Let $G/K$ be a compact and connected homogeneous space. 
Then, 
\begin{equation*}
\mathfrak g=\mathfrak k\oplus\mathfrak p,
\end{equation*}
where $\mathfrak g$ and $\mathfrak k$ are the Lie algebras of $G$ and $K$, respectively, and $\mathfrak p$ is called \emph{isotropy representation}. 
Let $Q$ be an $\Ad_{|_K}$-invariant scalar product on $\mathfrak p$.
For every $G$-invariant Riemannian metric $g$, $\mathfrak p$ decomposes into $\Ad_{|_K}$-invariant irreducible summands:
\begin{equation}\label{eqn:3}
\mathfrak p=\mathfrak p_1\oplus\mathfrak p_2\oplus\dots\oplus\mathfrak p_l,
\end{equation}
such that $g$ is diagonal with respect to $Q$:
\begin{equation}\label{eqn:4}
g=x_1Q_{|_{\mathfrak p_1}}\oplus x_2Q_{|_{\mathfrak p_2}}\oplus\dots\oplus x_lQ_{|_{\mathfrak p_l}},
\end{equation}
where $x_i>0$, for all $i=1,\dots,l$. 
In general, the decomposition \eqref{eqn:3} is not uniquely determined, but we do have uniqueness for the decomposition of $\mathfrak p$ into \emph{isotypical summands}. 
Each isotypical summand is given by the direct sum of irreducible summands of $\mathfrak p$ which are equivalent to a fixed summand. 
By Schur's lemma, every invariant metric $g$ and its Ricci tensor $\Ric(g)$ respect the splitting of $\mathfrak p$ into isotypical summands. 
We have the following definitions.
\begin{definition}
If the isotropy representation $\mathfrak p$ is irreducible as a $K$-representation, then the homogeneous space is called \emph{isotropy irreducible}.
\end{definition}
\begin{definition}
If the irreducible summands in \eqref{eqn:3} are pairwise inequivalent, then the isotropy representation is called \emph{monotypic}.
\end{definition}
\begin{remark}
In the monotypic case, by Schur's Lemma the Ricci tensor and the metric respect the splitting of $\mathfrak p$.
\end{remark}
Let $Q$ be an $Ad_{|_K}$-invariant scalar product on $\mathfrak p$ such that its restriction to every irreducible summand is a negative multiple of the Killing form $B$ of $G$. Consider now the decomposition \eqref{eqn:3} of $\mathfrak p$ into $\Ad_{|_K}$-invariant irreducible summands such that the $G$-invariant Riemannian metric $g$ diagonalises with respect to $Q$, as in \eqref{eqn:4}. 
Let $r_g$ be the Ricci endomorphism, which is defined as
\begin{equation*}
\Ric(g)(X,Y)=g(r_g(X),Y),
\end{equation*}
for all $X,Y\in\mathfrak p$. By \cite{Wang-Ziller:1} and \cite{Park-Sakane:1}, for every $i\in\{1,\dots,l\}$, the Ricci endomorphism on an isotypical summand $\mathfrak p_i$ is given by
\begin{equation*}
(r_g)_{|_{\mathfrak p_i}}=\Bigg(\frac{b_i}{2x_i}-\frac{1}{2d_i}\sum_{j,k=1}^l{[ijk]\frac{x_k}{x_ix_j}}+\frac{1}{4d_i}\sum_{j,k=1}^l{[ijk]\frac{x_i}{x_jx_k}}\Bigg)\identity_{|_{\mathfrak p_i}},
\end{equation*}
where $d_i=\dim(\mathfrak p_i)$ and $b_i$ is defined by
\begin{equation*}
-B_{|_{\mathfrak p_i}}=b_iQ_{|_{\mathfrak p_i}},
\end{equation*}
for all $i=1,\dots,l$. Moreover, the \emph{structure constants} $[ijk]$, which appear in the above formula, are defined by
\begin{equation*}
[ijk]=\sum{Q([e_\alpha,e_\beta],e_\gamma)^2},
\end{equation*}
where the sum is taken over the $Q$-orthonormal bases $\{e_\alpha\}_\alpha$, $\{e_\beta\}_\beta$ and $\{e_\gamma\}_\gamma$ of $\mathfrak p_i$, $\mathfrak p_j$ and $\mathfrak p_k$, respectively. Note that $[ijk]$ is symmetric in $i,j$ and $k$.

The relations between these quantities have been described in \cite{Wang-Ziller:1}:
\begin{equation}\label{eqn:5}
d_ib_i=2d_ic_i+\sum_{j,k=1}^l{[ijk]},\quad1\leq i\leq l,
\end{equation}
where $c_i$ are the nonnegative constants defined by:
\begin{equation*}
\mathcal C_{\mathfrak p_i,Q}=c_i\cdot\identity_{|_{\mathfrak p_i}},\quad1\leq i\leq l,
\end{equation*}
where $\mathcal C_{\mathfrak p_i,Q}=-\sum_\alpha{\ad(e_\alpha)\circ\ad(e_\alpha)}$ is the Casimir operator for the adjoint representation of $\mathfrak p_i$, for all $i=1,\dots,l$.

We refer the reader to \cite{Besse:1,Wang-Ziller:1} as good references on homogeneous spaces.
\subsection{Singularities in the Ricci flow}
We will now consider the Ricci flow equation in the homogeneous case. We note that, if $M$ is $G$-homogeneous, equation \eqref{eqn:1} becomes a system of nonlinear ordinary differential equations. More precisely, choose a background metric $Q$. Then consider the following decomposition of $\mathfrak p$:
\begin{equation*}
\mathfrak p=\mathfrak p_1\oplus\dots\oplus\mathfrak p_l,
\end{equation*}
where the $\mathfrak p_i$'s are pairwise inequivalent $\Ad_{|K}$-invariant irreducible summands, so that we are in the monotypic case. 
As every $G$-invariant Riemannian metric diagonalises as in \eqref{eqn:4}, if the initial metric is $G$-homogeneous, every solution to the Ricci flow equation will take the form
\begin{equation*}
g(t)=x_1(t)Q_{|_{\mathfrak p_1}}\oplus x_2(t)Q_{|_{\mathfrak p_2}}\oplus\dots\oplus x_l(t)Q_{|_{\mathfrak p_l}},
\end{equation*}
where $x_1(t),\dots,x_l(t)$ are smooth functions of $t$, which are strictly positive for all $t$ for which $g(t)$ is defined. We then have that $g(t)$ defined above is a Ricci flow if and only if $x_1(t),\dots,x_l(t)$ satisfy the system
\begin{equation}\label{eqn:6}
\dot{x}_i(t)=-b_i+\frac{1}{d_i}\sum_{j,k=1}^l{[ijk]\frac{x_k(t)}{x_j(t)}}-\frac{1}{2d_i}\sum_{j,k=1}^l{[ijk]\frac{x_i(t)^2}{x_j(t)x_k(t)}},
\end{equation}
for all $i=1,\dots,l$, and where $\dot{}$ indicates the derivative with respect to $t$, together with the condition $x_i(t)>0$, for all $i=1,\dots,l$.
\newline

We will now recall the definition of singular solution to the Ricci flow. 
\begin{definition}
Let $(M,g)$ be a closed Riemannian manifold.
A solution $g(t)$ to the Ricci flow \eqref{eqn:1} on $M\times[0, T)$, with $T\leq+\infty$, is a \emph{maximal solution} if either $T=+\infty$ or $T<+\infty$ and the norm of the curvature tensor $|\Rm(g(t))(x,t)|$ is unbounded, as $t\rightarrow T$. 
In the latter case, the maximal solution is called \emph{singular}. 
\end{definition}
According to Hamilton \cite{Hamilton:1}, we can classify singular solutions to the Ricci flow into type I and type II singularities. 
We say that a solution $g(t)$, with $t\in[0,T)$, to \eqref{eqn:1} develops a \emph{type I singularity} at $t=T$ if
\begin{enumerate}
\item $T<+\,\infty$,
\item $\sup_{t\in[0,T)}{\left(\sup_{p\in M}{|\Rm(g(t))|_{g(t)}(p,t)}\right)}=+\,\infty$,
\item $\sup_{t\in[0,T)}{\left((T-t)\sup_{p\in M}{|\Rm(g(t))|_{g(t)}(p,t)}\right)}<+\,\infty$.
\end{enumerate}
On the other hand, a solution $g(t)$ to \eqref{eqn:1}, with $t\in[0,T)$, develops a \emph{type II singularity} at $t=T$ if i) and ii) above are satisfied and if
\begin{equation*}
\sup_{t\in[0,T)}{\bigg((T-t)\sup_{p\in M}{|\Rm(g(t))|_{g(t)}(p,t)}\bigg)}=+\,\infty.
\end{equation*}

\subsubsection{Type I singularities in HRF}
Let $g(t)$ be a Ricci flow on $G/K\times[0,T)$, with $T<\infty$ and $g(0)$ $G$-invariant, and suppose that $T$ is a type I singularity for $g(t)$. 
Then, by \cite[Theorems 1.2-1.3]{Enders-Muller-Topping:1} and the homogeneity, as $t\rightarrow T$,
\begin{equation*}
|R(g(t))|\rightarrow+\infty
\end{equation*}
at the type I rate and, if $\Vol_{g(0)}(G/K)<\infty$, then
\begin{equation*}
\Vol_{g(t)}(G/K)\rightarrow0.
\end{equation*}

\subsection{Ancient solutions to the Ricci flow}
\begin{definition}
Let $(M,g)$ be a closed Riemannian manifold. 
A solution $g(t)$ to the Ricci flow \eqref{eqn:1} on $M$ which is defined on the time interval $(-\infty,T)$, with $T<\infty$, is called an \emph{ancient solution}.
\end{definition}
According to Hamilton \cite{Hamilton:1}, we can classify ancient solutions to the Ricci flow in type I and type II. 
We say that an ancient solution $g(t)$ to \eqref{eqn:1} is of \emph{type I} if
\begin{equation*}
\lim_{t\rightarrow-\infty}{\bigg(|t|\sup_{p\in M}{|\Rm(g(t))|_{g(t)}(p,t)}\bigg)}<\infty.
\end{equation*}
Otherwise, we say that the ancient solution is of \emph{type II}. 

These kind of solutions are important, because they arise as limits of blow ups of singular solutions to the Ricci flow near finite time singularities. 

\subsection{A notion of convergence for metric spaces}\label{subsec:1}
The \emph{Hausdorff-Gromov} distance is a way of measuring distances between metric spaces. Let $(A,d_A)$ and $(B,d_B)$ be two metric spaces. An \emph{$\epsilon$-approximation} between $A$ and $B$, denoted by $A\sim_\epsilon B$, is a subset $S\subseteq A\times B$ with the following properties:
\begin{enumerate}
\item both the projections of $S$ to $A$ and $B$ are onto,
\item for all $(p_1,p_2),(q_1,q_2)\in S$, $|d_A(p_1,q_1)-d_B(p_2,q_2)|<\epsilon$.
\end{enumerate}
The \emph{Hausdorff-Gromov distance} between $A$ and $B$ is defined as
\begin{equation*}
d_{\text{H-G}}(A,B)=\inf\{\epsilon|A\sim_\epsilon B\}.
\end{equation*}
If such an $\epsilon$ does not exist, we write $d_{H-G}(A,B)=\infty$. 
We say that a sequence of metric spaces $\{(A_n,d_{A_n})\}_n$ converges to $(A,d_A)$ in the Hausdorff-Gromov topology if $d_{\text{H-G}}(A_n,A)\rightarrow0$, as $n\rightarrow\infty$. 
An example of this kind of convergence is given by the following proposition.
\begin{proposition}\label{prop:2}
Let $G/K$ be a compact and connected homogeneous space. 
Suppose that there exists an intermediate Lie group $H$, with $G>H>K$. 
Let $\mathfrak g$, $\mathfrak h$ and $\mathfrak k$ be the Lie algebras of $G$, $H$ and $K$, respectively. 
Suppose that $\mathfrak h$ is $\Ad_{|_K}$-invariant and that every $G$-invariant Riemannian metric on $G/K$ a submersion metric
\begin{equation*}
H/K\rightarrow G/K\rightarrow G/H.
\end{equation*}
Then, if the fibre $H/K$ shrinks to a point, $G/K$ converges in the Hausdorff-Gromov sense to $G/H$.
\end{proposition}
\begin{proof}
Let $S\subseteq G/K\times G/H$ be given by
\begin{equation*}
\{(gK,gH),g\in G\}.
\end{equation*}
Clearly, $S$ projects onto both $G/K$ and $G/H$. 
Now, we can decompose $\mathfrak g$ into two different ways:
\begin{equation*}
\mathfrak k\oplus\mathfrak p=\mathfrak g=\mathfrak h\oplus\mathfrak q,
\end{equation*}
where $\mathfrak q$ and $\mathfrak p$ are orthogonal complements of $\mathfrak h$ and $\mathfrak k$ in $\mathfrak g$, respectively. 
Because of our assumption that every $G$-invariant Riemannian metric on $G/K$ is a submersion metric, the tangent space at each point of $G/K$ splits into vertical and horizontal subspaces:
\begin{equation*}
T_{gK}(G/K)=\mathcal V_{gK}\oplus \mathcal H_{gK},
\end{equation*}
where the vertical space $\mathcal V_{gK}\simeq\mathfrak q$ is the tangent subspace to the fibre $H/K$ and the horizontal space $\mathcal H_{gK}\simeq\mathfrak q^\perp$ is its orthogonal complement in $T_{gK}(G/K)\simeq\mathfrak p$.
 Moreover, the submersion map $gK\mapsto gH$ defines an isometry between $\mathcal H_{gK}$ and $T_{gH}(G/H)$ and distances in the $\mathfrak q$-direction shrink. 
Hence, given $(gK,gH),(hK,hH)\in S$,
\begin{equation*}
|d_{G/K}(gK,hK)-d_{G/H}(gH,hH)|\rightarrow0,
\end{equation*}
which proves convergence of $G/K$ to $G/H$ in the Hausdorff-Gromov sense.
\end{proof}
For further details on these matters we refer the reader to \cite{Burago-Burago-Ivanov:1}.

\section{Isotropy irreducible spaces}\label{sec:3}
Let $G/K$ be a compact and connected homogeneous space such that it is effective. 
Suppose that $G/K$ is isotropy irreducible.
By Schur's lemma, isotropy irreducible spaces carry a unique (up to rescaling) invariant Riemannian metric, which must also be Einstein (with positive scalar curvature). 
So, the HRF is just a rescaling of the initial metric. 
We can deduce the behaviour of the HRF directly from the flow equation as this will be instructive for more complicated examples. 
In the case where $G/K$ is isotropy irreducible, the system \eqref{eqn:6} becomes as follows. 
Let $d=\dim(\mathfrak p)$ and choose the $\Ad_{|_K}$-invariant background metric $Q$ to be $-\frac{1}{b}B$, where $B$ is the Killing form of $\mathfrak g$ and $b>0$. 
We then have that the only nonzero structure constant is given by $[111]$.
Then, $g(t)=x(t)Q_{|_{\mathfrak p}}$ is a solution to the Ricci flow equation if $x(t)$ satisfies
\begin{equation*}
\dot{x}(t)=-\left(b-\frac{[111]}{2d}\right)=-C,
\end{equation*}
where  $C>0$, by \eqref{eqn:5}. 
Hence,
\begin{equation*}
x(t)=C(T-t),
\end{equation*}
where $T=\frac{x(0)}{C}$ is the singular time of the HRF. 
As $t\rightarrow T$, $G/K$ shrinks to a point, which is a type I singularity. 
We also have that the solution exists for all $t<0$ and, as $t\rightarrow-\infty$, $G/K$ expands homothetically. 
This ancient solution is of type I. 
\begin{example}
Consider the $n$-dimensional sphere $(S^n,g_{_{S^n}})$ with constant sectional curvature $+1$. 
The metric $g_{_{S^n}}$ is Einstein with Einstein constant given by $n-1$. 
Hence, the solution to the Ricci flow equation with initial metric $g_{_{S^n}}$ is given by
\begin{equation*}
g(t)=(1-2(n-1)t)g_{_{S^n}}.
\end{equation*}
\end{example}

\section{When the isotropy group is not maximal}\label{sec:5}
In this section, we are going to consider the following class of compact and connected homogeneous spaces.
Let $G/K$ be a compact and connected homogeneous space which is also effective. 
Let $\mathfrak p$ be the isotropy representation of $K$ and suppose that it splits into two inequivalent irreducible $\Ad_{|_K}$-invariant summands
\begin{equation*}
\mathfrak p=\mathfrak p_1\oplus\mathfrak p_2.
\end{equation*}
Suppose that there exists an intermediate Lie group $H$, with $K<H<G$, and Lie algebra given by
\begin{equation*}
\mathfrak h=\mathfrak k\oplus\mathfrak p_.
\end{equation*}
In particular, $H/K$ is isotropy irreducible and every $G$-invariant Riemannian metric on $G/K$ is given by a fixed Riemannian submersion
\begin{equation}\label{eqn:23}
H/K\rightarrow G/K\rightarrow G/H,
\end{equation}
by rescaling the metric on the fibre and on the base.
Choose a background metric $Q$ such that it is a negative multiple of the Killing form on both $\mathfrak p_1$ and $\mathfrak p_2$. 
Then,
\begin{equation*}
[112]=0
\end{equation*}
and the Ricci flow equation for the one-parameter family of homogeneous Riemannian metrics
\begin{equation*}
g(t)=x_1(t)Q_{|_{\mathfrak p_1}}\oplus x_2(t)Q_{|_{\mathfrak p_2}}
\end{equation*}
is given by the following system of nonlinear ODEs:
\begin{align}
&\dot{x}_1(t)=-\bigg(b_1-\frac{[111]}{2d_1}-\frac{[122]}{d_1}\bigg)-\frac{[122]}{2d_1}\frac{x_1(t)^2}{x_2(t)^2},\label{eqn:18a}\\
&\dot{x}_2(t)=-\bigg(b_2-\frac{[222]}{2d_2}\bigg)+\frac{[122]}{d_2}\frac{x_1(t)}{x_2(t)},\label{eqn:19a}
\end{align}
together with the condition that $x_1(t),x_2(t)>0$.
\begin{remarks}
\smallskip\noindent \;

\smallskip\noindent 1.\; We note that, if $b_1=b_2$, $d_1=d_2$ and $[111]=[222]$, the above system of equations coincides with the system (7.2)-(7.3) of \cite{Bakas-Kong-Ni:1}. This is, for example, the case of $G$ simple and both $G/H$ and $H/K$ symmetric spaces such that $d_1=d_2$.

\smallskip\noindent 2.\; Observe that, if $[122]=0$, the system \eqref{eqn:18a}-\eqref{eqn:19a} reduces to
\begin{align*}
&\dot{x}_1(t)=-\bigg(b_1-\frac{[111]}{2d_1}\bigg),\\
&\dot{x}_2(t)=-\bigg(b_2-\frac{[222]}{2d_2}\bigg).
\end{align*}
This is the situation in which the universal cover of $G/K$ is given by a product of isotropy irreducible homogeneous spaces (cf. \cite[Theorem 2.1]{Wang-Ziller:1}). 
In this case, the analysis is just the same as the one performed in section \ref{sec:3} on each factor.
\end{remarks}
Motivated by the above remark, we will assume that $[122]>0$. 
 Now, let
\begin{align*}
&A=\frac{[122]}{2d_1},\\
&B=\frac{[122]}{d_2},\\
&C=b_1-\frac{[111]}{2d_1}-\frac{[122]}{d_1}=2c_1+\frac{[111]}{2d_1},\\
&D=b_2-\frac{[222]}{2d_2}=2c_2+\frac{[222]}{2d_2}+\frac{2}{d_2}[122].
\end{align*}
$A$, $B$ and $D$ are all strictly positive, because $[122],d_1,d_2>0$ and by \eqref{eqn:5}. 
Moreover, \eqref{eqn:5} also implies that $C\geq0$. 
Note that $C=0$ if and only if $c_1=0$ and $[111]=0$. This is the case of $\mathfrak p_1$ being a trivial 1-dimensional summand in the decomposition of $\mathfrak p$.
Using these quantities, the system \eqref{eqn:18a}-\eqref{eqn:19a} can be written as
\begin{align}
&\dot{x}_1(t)=-C-A\frac{x_1(t)^2}{x_2(t)^2},\label{eqn:18}\\
&\dot{x}_2(t)=-D+B\frac{x_1(t)}{x_2(t)}.\label{eqn:19}
\end{align}
It is useful to consider the following quantity. Let
\begin{equation*}
y(t)=\frac{x_1(t)}{x_2(t)}.
\end{equation*}
Note that \eqref{eqn:18}-\eqref{eqn:19} can be written as
\begin{align}
&\dot{x}_1(t)=-C-Ay(t)^2,\label{eqn:17}\\
&\dot{x}_2(t)=-D+By(t).\label{eqn:25}
\end{align}
Under the Ricci flow \eqref{eqn:18}-\eqref{eqn:19}, $y(t)$ evolves as follows:
\begin{equation}\label{eqn:20}
\dot{y}(t)=\frac{1}{x_2(t)}\big(-C+Dy(t)-(A+B)y(t)^2\big).
\end{equation}
Note that the equation
\begin{equation}\label{eqn:21}
C-Dy(t)+(A+B)y(t)^2=0
\end{equation}
has solutions if and only if the following inequality is satisfied:
\begin{equation*}
\bigg(b_2-\frac{[222]}{2d_2}\bigg)^2-4[122]\bigg(\frac{1}{2d_1}+\frac{1}{d_2}\bigg)\bigg(b_1-\frac{[111]}{2d_1}-\frac{[122]}{d_1}\bigg)\geq0.
\end{equation*}
By \cite{Wang-Ziller:1}, we know that the above inequality is satisfied if and only if $G/K$ admits a $G$-invariant Einstein metric. 
More precisely, $(x_1(t),x_2(t))$ defines a homogeneous Einstein metric on $G/K$ if and only if $\frac{x_1(t)}{x_2(t)}$ is a (positive) root of \eqref{eqn:21}.
Observe that the roots of \eqref{eqn:21} are always positive because the positivity of $A$, $B$, $C$ and $D$ implies that the sum and the product of the roots are positive.
In particular, $G/K$ carries at most two distinct $G$-invariant Einstein metrics, up to scaling. 
\begin{remark}\label{remark:1}
If $y(0)=y_i$, then we will get the trivial solution corresponding to $y(t)=y_i$, for all $i=1,2$.
We then have that $y_i$, with $i=1,2$, are fixed by the HRF.
Moreover, as we already mentioned in the introduction, $x_i(t)$, with $i=1,2$, just scale homothetically for all $t$ such that a solution exists.
\end{remark}
The following lemma holds.
\begin{lemma}\label{lemma:5}
$y(t)$ is monotonically increasing or decreasing under the HRF.
\end{lemma}
\begin{proof}
According to the number of roots of \eqref{eqn:21}, we have that only three possibilities can occur:
\begin{itemize}
\item Equation \eqref{eqn:21} has no solutions and 
\begin{equation*}
-C+Dy(t)-(A+B)y(t)^2
\end{equation*}
is always negative;
\item Equation \eqref{eqn:21} has a unique solution $\bar{y}$ such that 
\begin{equation*}
-C+Dy(t)-(A+B)y(t)^2=-(A+B)(y(t)-\bar{y})^2;
\end{equation*}
\item Equation \eqref{eqn:21} has two distinct solutions $y_1$ and $y_2$ such that 
\begin{equation*}
-C+Dy(t)-(A+B)y(t)^2=-(A+B)(y(t)-y_1)(y(t)-y_2).
\end{equation*}
\end{itemize}
By \eqref{eqn:20}, we have that, when \eqref{eqn:21} has no roots or only one root, $y(t)$ is monotonically decreasing for all $t$ such that a solution to the HRF exists. 
Now suppose that \eqref{eqn:21} has two distinct roots $y_1$ and $y_2$. 
Suppose without loss of generality that $y_1<y_2$. 
Then, when $y(t)<y_1$ and $y(t)>y_2$, $y(t)$ is monotonically decreasing in $t$. 
On the other hand, when $y_1<y(t)<y_2$, $y(t)$ is monotonically increasing in $t$.
Because of the uniqueness of the solution and of remark \ref{remark:1}, if $y(0)\neq y_i$, with $i=1,2$, then $y(t)\neq y_i$ for all $t$ such that a solution to the HRF exists. 
Hence, $y(t)$ is monotonic along every solution to the HRF.
\end{proof}

We now have to distinguish between two cases: $C>0$ and $C=0$. 
We will begin by considering the situation in which $C>0$ and then, at the end of the section, we will study the case $C=0$.
Let $G/K$ be as above and such that $C>0$. 
Then the following theorem holds.
\begin{theorem}\label{theo:7}
There exists $T<\infty$ such that there exists a unique solution to the HRF on $G/K$ which is defined on the maximal time interval $[0,T)$. Moreover, $T$ is a type I singularity and, as $t\rightarrow T$, one of the following singular behaviours occurs:
\begin{enumerate}
\item The whole space shrinks to a point in finite time.
\item The fibre $H/K$ in \eqref{eqn:23} shrinks to a point in finite time and the total space $G/K$ converges in the Hausdorff-Gromov topology to $G/H$.
\end{enumerate}
Moreover, a necessary condition for i) to happen is that $G/K$ carries $G$-invariant Einstein metrics.
\end{theorem}
\begin{proof}
Let
\begin{align*}
&f_1(x_1,x_2)=-C-A\frac{x_1^2}{x_2^2},\\
&f_2(x_1,x_2)=-D+B\frac{x_1}{x_2},
\end{align*}
be the functions defined by the right-hand side of \eqref{eqn:18}-\eqref{eqn:19}. Since $A$ and $B$ are strictly positive, $f_1$, $f_2$ and their derivatives with respect to $x_1$ and $x_2$ are continuous if and only if $(x_1,x_2)$ belongs to
\begin{equation*}
\mathbf D=\{(x_1,x_2)\in\mathbb R^2|x_2\neq 0\}.
\end{equation*}
We can then apply a standard theorem of ODEs, see for example \cite[Theorem 1.1]{Brauer-Nohel:1}, which says that, given any initial condition in $((x_1)_0,(x_2)_0)\in\mathbf D$, there exists a unique solution $(x_1(t),x_2(t))$ to \eqref{eqn:18}-\eqref{eqn:19} such that $x_1(t_0)=(x_1)_0,x_2(t_0)=(x_2)_0$ and which depends continuously on $t$ and the initial data. 
Moreover, the solution $(x_1(t),x_2(t))$ exists on any interval $I$ containing $t_0$ and such that $(x_1(t),x_2(t))\in\mathbf D$, for every $t\in I$.

From \eqref{eqn:18}, we have that $x_1(t)$ is decreasing in $t$ and
\begin{equation*}
\dot{x}_1(t)\leq-C,
\end{equation*} 
which integrated gives
\begin{equation*}
x_1(t)\leq-Ct+x_1(0).
\end{equation*}
Hence, there exists $T\leq\frac{x_1(0)}{C}<\infty$ such that the unique solution to the Ricci flow equation will be defined on the maximal time interval $[0,T)$, otherwise $x_1(t)$ becomes negative. 
In particular, $x_1(t)$ will approach a finite limit, as $t\rightarrow T$.

Now, lemma \ref{lemma:5} enables us to conclude that $x_2(t)$ approaches a limit, as $t\rightarrow T$. 
In fact, as $\dot{x}_2(t)$ is given by
\begin{equation*}
\dot{x}_2(t)=-D+By(t)
\end{equation*}
and $y(t)$ is monotone, so it approaches a limit in $[0,\infty)$, there exists $\bar{t}\leq T$ such that $x_2(t)$ is monotonically decreasing or increasing for all $t>\bar{t}$. 
Hence, it approaches a limit, as $t\rightarrow T$.
We also have that this limit cannot be $+\infty$. 
In fact, from \eqref{eqn:19} this would imply that $\dot{x}_2(t)\rightarrow-D$, which is a negative value. 
By the mean value theorem, this is a contradiction.
Hence, $x_2(t)$ has a finite limit, as $t\rightarrow T$.
Similarly, if $\dot{x}_2(t)\rightarrow+\infty$, $x_2(t)$ cannot tend to zero through positive values.
From \eqref{eqn:19}, we see that, if $x_2(t)\rightarrow0$ and the limit of $x_1(t)$ is nonzero, then $\dot{x}_2(t)\rightarrow+\infty$, which is a contradiction.
We can then conclude that, as $t\rightarrow T$, $y(t)$ approaches a finite limit and there are only two possible singular behaviours:
\begin{enumerate}
\item Both $x_1(t)$ and $x_2(t)$ tend to zero.\label{case:1c}
\item $x_1(t)$ tends to zero and $x_2(t)$ has a finite limit, which is strictly positive.\label{case:2c}
\end{enumerate}
We will now analyse these two singular behaviours separately.

Let us begin with case \ref{case:1c}). 
The singular time $T$ is characterised by the shrinking of the whole space to a point, as both $x_1(t)$ and $x_2(t)$ tend to zero, as $t\rightarrow T$. 
Then, there exist two positive integers $n_1$ and $n_2$ such that
\begin{equation}\label{eqn:2}
x_i(t)=k_i(T-t)^{n_i}+o((T-t)^{n_i}),\quad i=1,2,
\end{equation}
where $k_1$ and $k_2$ are positive coefficients. 
First of all, we observe that $x_1(t)/x_2(t)$ being bounded for all $t\in[0,T)$ implies
\begin{equation*}
n_1\geq n_2>0.
\end{equation*}
Then, by substituting \eqref{eqn:2} into \eqref{eqn:18}-\eqref{eqn:19}, we obtain that $n_1=n_2=1$, which means that $x_1(t)$ and $x_2(t)$ tend to zero linearly in $t$. 
Moreover, we also have that $k_1$ and $k_2$ have to satisfy the following system of nonlinear equations:
\begin{align*}
\frac{C}{k_1}+A\frac{k_1}{k_2^2}&=1,\\
\frac{D}{k_2}-B\frac{k_1}{k_2^2}&=1.
\end{align*}
The above system is satisfied if and only if
\begin{equation*}
\tilde{g}=k_1Q_{|_{\mathfrak p_1}}\oplus k_2Q_{|_{\mathfrak p_2}}
\end{equation*}
defines a $G$-invariant Einstein metric on $G/K$.
Finally, $T$ is a type I singularity, because $|\Rm(g(t))|_{g(t)}$ is asymptotically given by a rescaling of $|\Rm(g(0))|_{g(0)}$ by $(T-t)^{-1}$ times a positive constant on each summand.

We can now consider case \ref{case:2c}). Here, the singular time $T$ is characterised by the fact that $x_1(t)$ becomes zero, as $t\rightarrow T$. 
The Ricci flow then has to stop because the metric has collapsed on $\mathfrak p_1$. 
Then, there exist a positive integer $n$ such that
\begin{align*}
x_1(t)&=k_1(T-t)^{n}+o((T-t)^{n}),\\
x_2(t)&=k_2+o(1),
\end{align*}
where $k_1$ and $k_2$ are two positive constants. 
Substituting these expressions into \eqref{eqn:18}-\eqref{eqn:19}, we get that $n=1$, which means that $x_1(t)$ tends to zero linearly in $t$.
\begin{claim}\label{claim:3}
$T$ is a type I singularity.
\end{claim}
\begin{proof}[Proof of the claim]
We note that, asymptotically, on $\mathfrak p_1$, the Ricci flow is simply given by a rescaling of the initial metric by a constant times $(T-t)$, while on $\mathfrak p_2$ we only have a rescaling by a positive constant of the initial metric.
We then observe that the squared norm of the curvature tensor respects the splitting of the metric $g(t)$. Hence, asymptotically, $|\Rm(g(t))|_{g(t)}^2$ behaves like $(T-t)^{-2}$. 
This implies that $|\Rm(g(t))|_{g(t)}$ blows up to $+\infty$ like $(T-t)^{-1}$, when $t\rightarrow T$. This implies that $T$ is a type I singularity for the Ricci flow.
\end{proof}
Finally, by proposition \ref{prop:2}, we have that, as $t\rightarrow T$, $G/K$ converges in the Hausdorff-Gromov sense to $G/H$.
This concludes the proof of the theorem.
\end{proof}
We now want to investigate the existence of ancient solutions on $G/K$ as above and with $C>0$ and associate to each singular behaviour the corresponding subset of initial conditions. 
In order to do this, we have to distinguish between three different cases:
\begin{enumerate}
\renewcommand{\theenumi}{(\alph{enumi}}
\item $G/K$ carries two $G$-invariant Einstein metrics, up to scaling;
\item $G/K$ carries one $G$-invariant Einstein metric, up to scaling;
\item $G/K$ does not carry any $G$-invariant Einstein metric.
\newcounter{enumi_non_max}
\setcounter{enumi_non_max}{\value{enumi}}
\end{enumerate}
\begin{remark}
Note that in the case where $G/K$ does not admit $G$-invariant Einstein metrics, $C$ is always strictly positive.
\end{remark}
\subsection{Case (a)}
Let us first consider the case in which $G/K$ admits two non isometric $G$-invariant Einstein metrics, up to scaling. 
Let $y_1>0$ and $y_2>0$ correspond to the these two $G$-invariant Einstein metrics. 
In particular, $y_1$ and $y_2$ are solutions to the following equation:
\begin{equation}\label{eqn:26}
C-Dy+(A+B)y^2=0.
\end{equation}
Suppose without loss of generality that $y_2>y_1$. We then have that
\begin{equation*}
\begin{aligned}
\frac{\dot{x}_2(t)}{x_2(t)}&=\dot{y}(t)\frac{D-By(t)}{(A+B)(y(t)-y_1)(y(t)-y_2)}\\
&=\frac{\dot{y}(t)}{A+B}\bigg(-\frac{1}{y(t)-y_1}\bigg(B+\frac{D-By_2}{y_2-y_1}\bigg)+\frac{1}{y(t)-y_2}\frac{D-By_2}{y_2-y_1}\bigg).
\end{aligned}
\end{equation*}
We can now integrate this expression and obtain:
\begin{equation*}
\Lambda x_2(t)=|y(t)-y_1|^{-\frac{1}{A+B}\Big(B+\frac{D-By_2}{y_2-y_1}\Big)}|y_2-y(t)|^{\frac{1}{A+B}\frac{D-By_2}{y_2-y_1}},
\end{equation*}
where $\Lambda$ is a non negative constant. We then have a first integral for the Ricci flow \eqref{eqn:18}-\eqref{eqn:19}. This first integral is given by
\begin{equation}\label{eqn:24}
\Lambda=\frac{1}{x_2(t)}\left|\frac{x_1(t)}{x_2(t)}-y_1\right|^{-\frac{1}{A+B}\Big(B+\frac{D-By_2}{y_2-y_1}\Big)}\left|y_2-\frac{x_1(t)}{x_2(t)}\right|^{\frac{1}{A+B}\frac{D-By_2}{y_2-y_1}}.
\end{equation}
We will now consider three different possible initial conditions:
\begin{enumerate}
\renewcommand{\theenumi}{(a)(\arabic{enumi}}
\item $y(0)<y_1$,
\item $y_1<y(0)<y_2$,
\item $y(0)>y_2$,
\end{enumerate}
where $y(0)=\frac{x_1(0)}{x_2(0)}$.
We have the following proposition:
\begin{proposition}\label{prop:3}
The initial conditions \textup{(a)(1)}, \textup{(a)(2)} and \textup{(a)(3)} above are preserved under the HRF.
\end{proposition}
\begin{proof}
This is essentially the same argument as in the proof of lemma \ref{lemma:5}.
In fact, because of the uniqueness of the solution and remark \ref{remark:1}, if $y(0)\neq y_i$, then $y(t)\neq y_i$ for all $i=1,2$ and for all $t$ such that a solution to the HRF exists. 
\end{proof}
Note that this proposition also follows from the conserved quantity $\Lambda$.

We will now associate to each initial condition \textup{(a)(1)}, \textup{(a)(2)} and \textup{(a)(3)} the corresponding behaviour of the Ricci flow and, then, we will investigate the existence of ancient Ricci flows on $G/K$.

If the initial condition $y(0)$ lies between $y_1$ and $y_2$, then we can prove the following theorem, which is a generalisation of \cite[Theorem 7.1]{Bakas-Kong-Ni:1}.
\begin{theorem}\label{theo:4}
If $y_1<y(0)<y_2$, there exists a positive constant $T<\infty$ such that there exists a unique type I ancient solution to the Ricci flow \eqref{eqn:18}-\eqref{eqn:19} defined on $(-\infty,T)$. Moreover, $T$ is a type I singularity and, as $t\rightarrow T$, $G/K$ shrinks to a point. This solution flows the invariant Einstein metric corresponding to $y_2$ to the invariant Einstein metric corresponding to $y_1$, as $t$ goes from $T$ to $-\infty$.
\end{theorem}
\begin{proof}
We begin by noticing that theorem \ref{theo:7} implies that there exists a positive constant $T<\infty$ such that there exists a unique solution to the Ricci flow with initial condition $y(0)$ and defined on the maximal time interval $[0,T)$. 
Moreover, $x_1(t)\rightarrow0$, as $t\rightarrow T$. 
By proposition \ref{prop:3}, we have that $y_1<y(t)<y_2$, for all $t\in[0,T)$. 
Hence, we also have that $x_2(t)\rightarrow0$, as $t\rightarrow T$. This means that $G/K$ shrinks to a point, as $t$ approaches the singular time $T$. 
The evolution equation of $y(t)$ is given by
\begin{equation}\label{eqn:22}
\dot{y}(t)=-\frac{(A+B)}{x_2(t)}(y(t)-y_1)(y(t)-y_2).
\end{equation}
Hence, $y(t)$ is increasing in $t$. 
By the proof of theorem \eqref{theo:7}, we have that both $x_1(t)$ and $x_2(t)$ tend to zero linearly in $t$. 
Recall that, as we approach the singular time $T$, $y(t)$ is increasing in $t$ and it approaches a limit. 
Then, near the singular time
\begin{align*}
&x_1(t)=k_1(T-t)+o((T-t)),\\
&x_2(t)=k_2(T-t)+o(T-t),
\end{align*}
where $k_1$ and $k_2$ are two positive coefficients. 
Substituting these expressions in \eqref{eqn:18}-\eqref{eqn:19} and taking the limit as $t\rightarrow T$, we obtain
\begin{align*}
&-k_1=-C-A\frac{k_1^2}{k_2^2},\\
&-k_2=-D+B\frac{k_1}{k_2},
\end{align*}
which means that $k_1Q_{|_{\mathfrak p_1}}\oplus k_2Q_{|_{\mathfrak p_2}}$ corresponds to a $G$-invariant Einstein metric on $G/K$. 
This implies that
\begin{equation*}
\lim_{t\rightarrow T}{y(t)}=y_2.
\end{equation*}
We now have to show the existence of ancient solutions to the Ricci flow. In order to do this, it is convenient to change the time parameter from $t$ to $\tau=-t$. Let $'$ denote the derivative with respect to $\tau$. Then the system \eqref{eqn:18}-\eqref{eqn:19} becomes
\begin{align}
x'_1(\tau)&=C+Ay(\tau)^2,\label{eqn:28}\\
x_2'(\tau)&=D-By(\tau),\label{eqn:29}
\end{align}
together with the condition $x_1(\tau)>0$ and $x_2(\tau)>0$. The evolution equation of $y(\tau)$ becomes:
\begin{equation}\label{eqn:30}
y'(\tau)=\frac{A+B}{x_2(\tau)}(y(\tau)-y_1)(y(\tau)-y_2).
\end{equation}
By proposition \ref{prop:3}, $y_1<y(\tau)<y_2$ along any solution to the Ricci flow. 
Hence, $y(\tau)$ is decreasing in $\tau$.
Moreover, the right-hand side of \eqref{eqn:28}-\eqref{eqn:29} are bounded for all $\tau$ such that a solution to the Ricci flow exists and the bounds depend on $y_2$. 
From \eqref{eqn:28}, we see that $x_1'(\tau)>0$ for all $\tau$. 
We also have that
\begin{equation*}
\begin{aligned}
x_2'(\tau)>D-By_2>D-\frac{BD}{A+B}>0
\end{aligned}
\end{equation*}
where we have used the fact that
\begin{equation*}
y_2=\frac{D+\sqrt{D^2-4C(A+B)}}{2(A+B)},
\end{equation*}
as it is the biggest solution to \eqref{eqn:21}. 
Hence, both $x_1(\tau)$ and $x_2(\tau)$ are increasing in $\tau$ with bounded derivatives. 
Moreover, from \eqref{eqn:30}, we have that $y(\tau)$ is decreasing in $\tau$. 
We can then apply standard ODE theory and conclude that a solution to the Ricci flow exists for all $\tau>0$, as long as $y(\tau)>y_1$. 
By the uniqueness of the solution to the Ricci flow equation, $y(\tau)$ cannot become $y_1$, as long as $x_1(\tau)$ and $x_2(\tau)$ are positive. 
Hence, we can conclude that for every initial condition $y_1<y(0)<y_2$, there exists a unique ancient solution to the Ricci flow defined for $\tau\in[0,+\infty)$, i.e. for $t\in(-\infty,0]$. 
To finish the proof we notice that, by the asymptotic analysis and the fact that $y(\tau)$ remains bounded, we have that both $x_1(\tau)$ and $x_2(\tau)$ increase linearly in $\tau$, as $\tau\rightarrow+\infty$. 
This implies that the ancient solution is of type I and that, as $\tau\rightarrow+\infty$, $y(\tau)\rightarrow y_1$.
\end{proof}
To conclude the study of the Ricci flow in this case, we will consider the other two possible initial conditions \textup{(a)(1)} and \textup{(a)(3)}. The following two theorems hold.
\begin{theorem}\label{theo:2}
If $y(0)<y_1$, there exists a positive constant $T<\infty$ such that there exists a unique type I ancient solution to the Ricci flow equation defined on $(-\infty,T)$. As $t\rightarrow T$, the fibre $H/K$ in \eqref{eqn:23} shrinks to a point and $G/K$ collapses in the Hausdorff-Gromov sense to $G/H$. Furthermore, as $t\rightarrow -\infty$, $y(t)\rightarrow y_1$.
\end{theorem}
\begin{proof}
As the initial condition is preserved under the Ricci flow, equation \eqref{eqn:22} implies that $y(t)$ is decreasing in $t$. 
By theorem \ref{theo:7}, there exists $T<\infty$ such that there exists a unique solution to the HRF which is defined on the maximal time interval $[0,T)$. 
Moreover, as we approach the singular time, $y(t)\rightarrow0$. 
In fact, suppose that $y(t)\rightarrow y_0$, as $t\rightarrow T$, where $y_0>0$. 
This would be the case of both $x_1(t)$ and $x_2(t)$ going to zero, as $t\rightarrow T$. 
From the proof of theorem \ref{theo:7}, we know that $x_1(t)$ and $x_2(t)$ tend to zero linearly in $t$. 
Using this fact and \eqref{eqn:17}-\eqref{eqn:25}, we can compute that $y_0<y_1$ has to be a $G$-invariant Einstein metric on $G/K$. 
However, this cannot happen, because $G/K$ carries exactly two homogeneous Einstein metrics, which correspond to $y_1$ and $y_2$. 
We can then conclude that, as $t\rightarrow T$, $y(t)\rightarrow0$, which means that $x_1(t)$ tends to zero, while $x_2(t)$ remains strictly positive. 
This tells us that the singular behaviour which characterises the HRF in this case is the shrinking of the fibre $H/K$ in \eqref{eqn:23} and the collapsing of $G/K$ to $G/H$ in the Hausdorff-Gromov sense.

We will now show the existence of ancient solutions to the Ricci flow. As we did in the proof of theorem \ref{theo:4}, let us change time parameter from $t$ to $\tau=-t$. 
Then, from \eqref{eqn:30}, we have that $y(\tau)$ is increasing in $\tau$. 
As $y(\tau)<y_1$ for all $\tau$ such that a solution to the above system exists, the derivatives $x_1'(\tau)$ and $x_2'(\tau)$ remain bounded. 
We then have that, as $\tau$ increases, the solution to the Ricci flow exists as long as $x_1(\tau)$ and $x_2(\tau)$ remain positive and $y(\tau)<y_1$. 
As $x_1(\tau)$ and $x_2(\tau)$ are increasing in $\tau$, the solution to the Ricci flow exists as long as $y(\tau)<y_1$. By the uniqueness of the solution, $y(\tau)$ cannot reach $y_1$ as long as $x_1(\tau)$ and $x_2(\tau)$ remain positive. 
Hence, the solution exists for all $\tau>0$. 
Moreover, using the asymptotic analysis and the fact that $y(\tau)$ remains bounded, it is possible to compute that $x_1(\tau)$ and $x_2(\tau)$ both increase linearly in $\tau$, as $\tau\rightarrow+\infty$. 
This fact implies that the ancient solution is of type I and, using \eqref{eqn:28} and \eqref{eqn:29}, it also implies that $y(\tau)\rightarrow y_1$, as $\tau\rightarrow+\infty$. 
This concludes the proof of the theorem.
\end{proof}
\begin{theorem}\label{theo:3}
If $y(0)>y_2$, there exists a positive constant $T<\infty$ such that there exists a unique solution to the Ricci flow equation defined on $[0,T)$. 
As $t\rightarrow T$, $G/K$ shrinks to a point and $y(t)\rightarrow y_2$. 
In particular, there are no ancient solutions to the HRF in this case.
\end{theorem}
\begin{proof}
By \eqref{eqn:22} and the fact that the initial condition is preserved under the Ricci flow, we have that $y(t)$ is decreasing in $t$. 
By theorem \ref{theo:7}, there exists a positive constant $T<\infty$ such that there exists a unique solution to the Ricci flow with initial condition given by $y(0)$ and defined on the maximal time interval $[0,T)$. 
Theorem \ref{theo:7} also tells us that, as $t\rightarrow T$, $x_1(t)\rightarrow0$. Moreover, as $y(t)>y_2$ for all $t\in[0,T)$, we have that $x_1(t)\rightarrow0$ implies that $x_2(t)\rightarrow0$, as $t\rightarrow T$. 
Hence, the behaviour of the Ricci flow, as $t$ approaches the singular time $T$, is given by the shrinking of the whole space to a point in finite time. 
Moreover, by the proof of theorem \ref{theo:7}, we know that $x_1(t)$ and $x_2(t)$ tend to zero linearly in $t$. Using this fact, from \eqref{eqn:18}-\eqref{eqn:19}, we can compute that $y(t)\rightarrow y_2$, as $t\rightarrow T$. 

It remains to show that, with this initial condition, there are no ancient solutions to the HRF. 
Let us change time parameter from $t$ to $\tau=-t$. 
We are then considering the system given by \eqref{eqn:28}-\eqref{eqn:29}. 
By \eqref{eqn:30}, $y(\tau)$ is increasing in $\tau$. 
The evolution equation \eqref{eqn:28} implies that $x_1(\tau)$ is increasing in $\tau$. 
Moreover, $x_1'(\tau)>C>0$.
We then need to understand the behaviour of $x_2(\tau)$, for all $\tau$ such that a solution exists. 
If $x_2'(0)<0$, by \eqref{eqn:29} $x_2'(\tau)<0$ for all $\tau$ such that a solution exists. 
If $x_2'(0)$ is non negative, then $x'_2(\tau)$ will be positive until $\tau=\tau_0$ such that $x_2'(\tau_0)=0$. 
In fact, $x_2'(\tau)=0$ if and only if $y(\tau)=\frac{D}{B}>y_2$. 
Moreover, if $y(0)<\frac{D}{B}$, then $y(\tau)$ will become $\frac{D}{B}$ in finite time. 
In fact, if $y(\tau)$ approaches a limit, then this limit has to correspond to a homogeneous Einstein metric on $G/K$, but this is not possible because the only two invariant Einstein metrics are $y_1$ and $y_2$ and $y(\tau)>y_2$, for all $\tau$ such that a solution exists.
As 
\begin{equation*}
x''_2(\tau)=-By'(\tau)<0,
\end{equation*}
for all $\tau$, $\tau_0$ is the maximum point of $x_2(\tau)$.
Then, $x_2'(\tau)<0$ for all $\tau>\tau_0$. 
If a solution to \eqref{eqn:28}-\eqref{eqn:29} existed for all $\tau>0$, then $x_1(\tau)$ would diverge to $+\infty$ and $x_2(\tau)$ would remain bounded and positive. 
From \eqref{eqn:29}, this implies that $x'_2(\tau)\rightarrow-\infty$, as $\tau\rightarrow+\infty$, which is not possible. 
Hence, the solution will have to stop at a finite $\bar{\tau}$, which is characterised by $x_2(\tau)$ becoming zero.
This concludes the proof of the theorem.
\end{proof}

\subsection{Case (b)}
We will now consider the case in which $G/K$ admits exactly one $G$-invariant Einstein metric, up to scaling. Let $\bar{y}$ be the unique solution to \eqref{eqn:26}. We then have that
\begin{equation*}
\frac{\dot{x}_2(t)}{x_2(t)}=\frac{\dot{y}(t)}{A+B}\frac{D-y(t)}{(y(t)-\bar{y})^2},
\end{equation*}
which integrated gives the first integral
\begin{equation}\label{eqn:31}
\widetilde{\Lambda}=\frac{1}{x_2(t)}\exp{\bigg(-\frac{1}{A+B}\frac{D-y(t)}{y(t)-\bar{y}}\bigg)}|y(t)-\bar{y}|^{-\frac{1}{A+B}},
\end{equation}
where $\widetilde{\Lambda}$ is a positive constant. 

We will now describe the behaviour of the Ricci flow, according to the initial condition. 
We can have two possible initial conditions:
\begin{enumerate}
\renewcommand{\theenumi}{(b)(\arabic{enumi}}
\item $y(0)<\bar{y}$,
\item $y(0)>\bar{y}$.
\end{enumerate}
\begin{remark}
We note that the evolution equation of $y(t)$ is given by
\begin{equation*}
\dot{y}(t)=-\frac{A+B}{x_2(t)}(y(t)-\bar{y})^2.
\end{equation*}
Hence, $y(t)$ is always decreasing in $t$.
\end{remark}
Using the first integral \eqref{eqn:31} and the fact that the solution to the HRF is unique we can prove the following proposition, which is an analogue of proposition \ref{prop:3}.
\begin{proposition}
The initial conditions \textup{(b)(1)} and \textup{(b)(2)} above are preserved under the Ricci flow.
\end{proposition}

We then have that the following theorem holds:
\begin{theorem}
If $y(0)<\bar{y}$, there exists a positive constant $T<\infty$ such that there exists a unique type I ancient solution to the Ricci flow \eqref{eqn:18}-\eqref{eqn:19} with initial condition $y(0)$ and defined on the maximal time interval $(-\infty,T)$. As $t\rightarrow T$, the fibre $H/K$ in \eqref{eqn:23} shrinks to a point and $G/K$ collapses to $G/H$ in the Hausdorff-Gromov sense. Moreover, as $t\rightarrow-\infty$, $y(t)\rightarrow\bar{y}$.

On the other hand, if $y(0)>\bar{y}$, there exists a positive constant $T<\infty$ such that there exists a unique solution to the Ricci flow with initial metric given by $y(0)$ and defined on the maximal time interval $[0,T)$. 
As $t\rightarrow T$, the whole space $G/K$ shrinks to a point and $y(t)\rightarrow\bar{y}$. 
In particular, there are no ancient solutions to the Ricci flow with initial condition $y(0)$.
\end{theorem}
To prove the first part of the theorem, we proceed as in the proof of theorem \ref{theo:2} and the proof of the second part of the theorem works in the same way as the one of theorem \ref{theo:3}. As these proofs are very similar, we will omit it.

\subsection{Case (c)}
In this case, by theorem \ref{theo:7}, the behaviour of the Ricci flow starting at any $G$-invariant Riemannian metric will be characterised by the shrinking of the fibre $H/K$ in \eqref{eqn:23} and the collapsing in the Hausdorff-Gromov sense of $G/K$ to $G/H$. 
Moreover, the following proposition is true.
\begin{proposition}
If $G/K$ described above does not carry any $G$-invariant Einstein metric, then the Ricci flow on $G/K$ starting at any homogeneous Riemannian metric does not have ancient solutions.
\end{proposition}
The proof of this proposition is very similar to the non-existence proof of theorem \ref{theo:3}, so we will omit it.

Note that this case corresponds to homogeneous spaces which do not carry any invariant Einstein metric, see \cite{Wang-Ziller:1}.
In particular, if under the homogeneous Ricci flow the space can never collapse to a point, then it does not carry any invariant Einstein metric.
\begin{example}[\cite{Wang-Ziller:1}]
An example of a homogeneous space which corresponds to case (c) is given by $SU(4)/SU(2)$.
This example was first considered by Wang and Ziller in \cite{Wang-Ziller:1}.
In this case, $G=SU(4)$ and $K=SU(2)$.
Moreover, we have that there exists only one intermediate Lie group which is given by $H=Sp(2)$.
Both $G/H$ and $H/K$ are isotropy irreducible and every $SU(4)$-invariant Riemannian metric on $G/K$ is obtained from a submersion metric
\begin{equation*}
Sp(2)/SU(2)\rightarrow SU(4)/SU(2)\rightarrow SU(4)/Sp(2).
\end{equation*}
We can choose the negative of the Killing form $B$ as a background metric, so that $b_1=b_2=1$.
From \cite{Wang-Ziller:1}, we have that $d_1=7$, $d_2=5$ and the only nonzero structure constants are given by
\begin{equation*}
[111]=\frac{21}{20},\quad[122]=\frac{7}{4}.
\end{equation*}
Then,
\begin{equation*}
A=\frac{1}{8},\,\,B=\frac{7}{20},\,\,C=\frac{27}{40},\,\,D=1.
\end{equation*}
So equation \eqref{eqn:26} does not have any root.
This means that, under the homogeneous Ricci flow, the fibre $Sp(2)/SU(2)$ always shrinks to a point in finite time and $SU(4)/SU(2)$ collapses in the Hausdorff-Gromov sense to $SU(4)/Sp(2)$.
In particular, this shows that the space does not carry any $SU(4)$-invariant Einstein metric, which was first proved in \cite{Wang-Ziller:1}.
\end{example}

To conclude this section, we will consider $G/K$ as described above and such that $C=0$. 
In this case, equation \eqref{eqn:20} becomes
\begin{equation*}
\dot{y}(t)=\frac{1}{x_2(t)}(Dy(t)-(A+B)y(t)^2).
\end{equation*}
Hence, $G/K$ always admits a $G$-invariant Einstein metric which is unique up to scaling.
Let
\begin{equation*}
\bar{y}=\frac{D}{A+B}
\end{equation*} 
denote this homogeneous Einstein metric. 
Then the evolution equation of $y(t)$ becomes
\begin{equation*}
\dot{y}(t)=-(A+B)\frac{y(t)}{x_2(t)}(y(t)-\bar{y}).
\end{equation*}
We have two possible initial conditions: $y(0)<\bar{y}$ or $y(0)>\bar{y}$. 
These are preserved under the HRF, because we have uniqueness of the solution. 
Then, if $y(0)>\bar{y}$, $y(t)$ will be decreasing in $t$ and it will be increasing in $t$ if $y(0)<\bar{y}$. 
Moreover, $\bar{y}$ is fixed point of the HRF, so the regions $\{y>\bar{y}\}$ and $\{y<\bar{y}\}$ are preserved.
We can prove the following theorem.
\begin{theorem}
If $y(0)>\bar{y}$, there exists $T<\infty$ such that there exists a unique solution to the HRF on $G/K$ which is defined on the maximal time interval $[0,T)$. 
$T$ is a type I singularity and, as $t\rightarrow T$, $G/K$ shrinks to a point and $y(t)$ approaches the unique invariant Einstein metric $\bar{y}$. 
Moreover, there are no ancient solutions.

Whereas, if $y(0)<\bar{y}$, there exists $T<\infty$ such that there exists a unique type II ancient solution to the HRF on $G/K$ which is defined on $(-\infty,T)$.
$T$ is a type I singularity and, as $t\rightarrow T$, $G/K$ shrinks to a point and $y(t)$ approaches $\bar{y}$. 
Finally, as $t\rightarrow-\infty$, $y(t)\rightarrow0$ and $G/K$ collapses in the Hausdorff-Gromov sense to $G/H$.
\end{theorem}
\begin{proof}
The HRF in this case corresponds to the following system of nonlinear ODEs:
\begin{align*}
&\dot{x}_1(t)=-Ay^2(t),\\
&\dot{x}_2(t)=-D+By(t).
\end{align*}
We will first show that the HRF always develops a singularity in finite time. 
In fact, if $y(0)<\bar{y}$, then $y(t)<\bar{y}$, which implies that
\begin{equation*}
\dot{x}_2(t)<-D+B\bar{y}<0.
\end{equation*}
Hence the HRF has to stop before $x_2(t)$ becomes zero. 
Moreover, $x_2(t)\rightarrow0$ implies that $x_1(t)\rightarrow0$, because $y(t)=\frac{x_1(t)}{x_2(t)}<\bar{y}$, for all $t$ such that a solution exists. 
On the other hand, if $y(0)>\bar{y}$, we have that
\begin{equation*}
\dot{x}_1(t)<-A\bar{y}^2<0,
\end{equation*}
so the HRF will have to stop before $x_1(t)$ becomes zero. 
This implies that also $x_2(t)$ tends to zero, because $y(t)>\bar{y}$, for all $t$ such that a solution exists. 
Hence, every solution to the HRF will develop a singularity in a finite time $T$ and, as $t\rightarrow T$, $y(t)\rightarrow\bar{y}$. 
Moreover, we can classify this singularity to be of type I, because both $x_1(t)$ and $x_2(t)$ tend to zero linearly in $t$.

We will now investigate the existence of ancient solutions. 
As in the proof of theorem \ref{theo:3}, we can show that there are no ancient solutions when $y(0)>\bar{y}$, because $x_2(t)$ becomes zero in finite time. 
Whereas, if $y(0)<\bar{y}$, both $x_1(\tau)$ and $x_2(\tau)$ are increasing with bounded derivatives for all $\tau\geq0$. 
Hence, there always exists a unique ancient solution to the HRF with initial condition $y(0)$. 
We can then classify this ancient solution as of type II. 
This is due to the fact that, as $t\rightarrow-\infty$, $y(t)\rightarrow0$, which implies that $x_1(t)$ increases logarithmically in $t$, while $x_2(t)$ increases linearly in $t$. 
Hence,
\begin{equation*}
|\Rm(g(t))|_{g(t)}^2\geq\frac{1}{|t|},
\end{equation*}
as $t\rightarrow-\infty$. 
This implies that the ancient solution is of type II.
\end{proof}

\subsection{The scalar curvature}
We are now going to study the behaviour of the scalar curvature under the system \eqref{eqn:18}-\eqref{eqn:19}.
The scalar curvature of $g(t)$ is given by
\begin{equation*}
R(t)=\frac{1}{x_1(t)}\bigg(C\frac{d_1}{2}+D\frac{d_2}{2}\frac{x_1(t)}{x_2(t)}-A\frac{d_1}{2}\frac{x_1(t)^2}{x_2(t)^2}\bigg),
\end{equation*}
We also have the following relation between $A$ and $B$:
\begin{equation}\label{eqn:411}
\frac{d_1}{2}A=\frac{d_2}{4}B.
\end{equation}

We will begin by identifying the regions in our phase space where $R$ is positive or negative. 
The scalar curvature vanishes when
\begin{equation}\label{eqn:410}
C\frac{d_1}{2}+D\frac{d_2}{2}y-A\frac{d_1}{2}y^2=0,
\end{equation}
where $y=\frac{x_1}{x_2}$ and $x_1,x_2$ are standard coordinates in $\mathbb{R}^2$, with $x_1$ as the vertical axis and $x_2$ as the horizontal axis.
The discriminant of \eqref{eqn:410} is given
\begin{equation*}
\frac{d_2^2}{4}D^2+d_1^2AC,
\end{equation*}
which is always strictly positive.
So, \eqref{eqn:410} has always two different solutions, which are given by
\begin{align*}
\bar{y}_1&=\frac{1}{d_1C}\bigg(-\frac{d_2}{2}D+\bigg(\frac{d_2^2}{4}D^2+d_1^2AC\bigg)^{\frac{1}{2}}\bigg),\\
\bar{y}_2&=\frac{1}{d_1C}\bigg(-\frac{d_2}{2}D-\bigg(\frac{d_2^2}{4}D^2+d_1^2AC\bigg)^{\frac{1}{2}}\bigg),
\end{align*}
in the case where $C>0$, and by
\begin{align*}
\tilde{y}_1&=0,\\
\tilde{y}_2&=\frac{d_2}{d_1}\frac{D}{A},
\end{align*}
in the case where $C=0$.
Clearly, $\bar{y}_2$ is negative and $\tilde{y}_2$ is positive.
We also have that $\bar{y}_1$ is positive.
In fact,
\begin{equation*}
\bar{y}_1>\frac{1}{d_1C}\bigg(-\frac{d_2}{2}D+\frac{d_2}{2}D\bigg)=0.
\end{equation*}
The solution to \eqref{eqn:410} define lines through the origin in $\mathbb{R}^2$, which separates the regions in which $R$ is positive and the regions in which $R$ is negative.
The following two pictures illustrate this in the two cases in which $C>0$ and $C=0$:
\begin{center}
\begin{tikzpicture}[>=stealth,scale=.4]
\draw [very thin,->] (-6,0) -- (6,0) node [fill=white,below=1pt] {$\scriptstyle x_2$} coordinate (x axis);
\draw [very thin,->] (0,-6) -- (0,6) node [fill=white,right=1pt] {$\scriptstyle x_1$} coordinate (y axis);
\draw (-3,-6) -- (3,6) (-6,2.5) -- (6,-2.5);
\draw (3,6.2) node [right] {$\scriptstyle y=\bar y_1$} (-6,2.8) node [left] {$\scriptstyle y=\bar y_2$};
\draw (1,4) node [above] {$\scriptstyle R<0$} (-1,-4) node [below] {$\scriptstyle R>0$};
\draw (-4,-2) node [left] {$\scriptstyle R<0$} (4,2) node [right] {$\scriptstyle R>0$};
\draw (-4,0.5) node[right] {$\scriptstyle R>0$} (-3,4.5) node [left] {$\scriptstyle R<0$};
\draw (1,-3.5) node [right] {$\scriptstyle R>0$} (5,-0.7) node [left] {$\scriptstyle R<0$};
\draw (0,-7) node[below] {$C>0$};
\end{tikzpicture}
\begin{tikzpicture}[>=stealth,scale=.4]
\draw [very thin,->] (-6,0) -- (6,0) node [fill=white,below=1pt] {$\scriptstyle x_2$} coordinate (x axis);
\draw [very thin,->] (0,-6) -- (0,6) node [fill=white,right=1pt] {$\scriptstyle x_1$} coordinate (y axis);
\draw (-3,-6) -- (3,6);
\draw (-3,3) node [above]{$\scriptstyle R<0$};
\draw (3,-3) node[above]{$\scriptstyle R>0$};
\draw (3,6.2) node [right]{$\scriptstyle y=\tilde y_2$};
\draw (1,4) node [above] {$\scriptstyle R<0$} (-1,-4) node [below] {$\scriptstyle R>0$};
\draw (-4,-2) node [left] {$\scriptstyle R<0$} (4,2) node [right] {$\scriptstyle R>0$};
\draw (0,-7) node[below] {$C=0$};
\end{tikzpicture}
\end{center}

We will study the behaviour of the scalar curvature in the first quadrant, where both $x_1(t)$ and $x_2(t)$ are positive.
We will distinguish the cases in which $C>0$ and $C=0$.

Let us begin with the case in which $C>0$.
Here we have that $R>0$ if and only if $y<\bar{y}_1$.
From \eqref{eqn:19}, we have that the region in which $x_2(t)$ is decreasing in $t$ is given by
\begin{equation*}
\left\{(x_1,x_2)\in\mathbb{R}^2|x_1<\frac{D}{B}x_2\right\}.
\end{equation*}
We will now show that in this region $R$ is positive.
If $y=\frac{D}{B}$,
\begin{equation*}
R=\frac{1}{x_1}\bigg(\frac{d_1}{2}C+\frac{d_2}{2}\frac{D^2}{B}-\frac{d_1}{2}\frac{D^2}{B^2}A\bigg).
\end{equation*}
By \eqref{eqn:411},
\begin{equation*}
R=\frac{1}{x_1}\bigg(\frac{d_1}{2}C+\frac{d_2}{2}\frac{D^2}{B}-\frac{d_2}{4}\frac{D^2}{B}\bigg)=\frac{1}{x_1}\bigg(\frac{d_1}{2}C+\frac{d_2}{4}\frac{D^2}{B}\bigg)>0.
\end{equation*}
Hence, on the boundary of the region where $x_2(t)$ is decreasing in $t$, $R>0$.
This means that
\begin{equation*}
\frac{D}{B}<\bar{y}_1,
\end{equation*}
which implies that $R$ is positive for every $y<\frac{D}{B}$.
We can easily see the the region $\{y<\frac{D}{B}\}$ is invariant for the system \eqref{eqn:18}-\eqref{eqn:19}.
In fact, the vector field $(\dot{x}_1(t),\dot{x}_2(t))$ on the boundary $\{y=\frac{D}{B}\}$ is given by $(\dot{x}_1(t),0)$, which points towards the interior, as $\dot{x}_1(t)$ is negative.
Note that this region is characterised by $x_2(t)$ decreasing in $t$.
So, this invariant region is located where $R$ is positive.
Moreover, if $y>\frac{D}{B}$, $x_2(t)$ is increasing and the solution is moving towards the boundary of the invariant region.
In theorem \ref{theo:7}, we proved that the solution either cross the $x_2$-axis or goes to the origin in finite time.
This means that the trajectory has to enter the invariant region $\{y<\frac{D}{B}\}$ in finite time.
We then know that if $R$ is negative initially, it will turn positive in finite time, because the solution has to enter the invariant region.

We will now consider the case in which $C=0$.
Here, $R>0$, when $y<\tilde{y}_2$.
By \eqref{eqn:411},
\begin{equation*}
\tilde{y}_2=\frac{d_2}{d_1}\frac{D}{A}=2\frac{d_2}{d_1}\frac{D}{d_2B}d_1=2\frac{D}{B}
\end{equation*}
So, the invariant region in which $x_2(t)$ is decreasing is located where $R(t)$ is positive.
This means that if the scalar curvature is negative initially, then it has to turn positive in finite time.

\section{When the isotropy group is maximal}\label{sec:1}
In this section, we are going to study compact and connected Lie groups $G/K$ such that $G/K$ is effective.
Suppose that $K$ is maximal in $G$ and the isotropy representation of $K$ decomposes into two inequivalent irreducible $\Ad_{|_K}$-invariant summands:
\begin{equation*}
\mathfrak p=\mathfrak p_1\oplus\mathfrak p_2.
\end{equation*}
By our assumptions, the structure constants $[112]$ and $[122]$ are both nonzero. 
Then,
\begin{equation*}
g(t)=x_1(t)Q_{|_{\mathfrak p_1}}\oplus x_2(t)Q_{|_{\mathfrak p_2}}
\end{equation*}
is a solution to the Ricci flow on $G/K$ if and only if $x_1(t)$ and $x_2(t)$ satisfy the following system of nonlinear ODEs:
\begin{align}
\dot{x}_1(t)&=-\bigg(b_1-\frac{[111]}{2d_1}-\frac{[122]}{d_1}\bigg)+\frac{[112]}{d_1}\frac{x_2(t)}{x_1(t)}-\frac{[122]}{2d_1}\frac{x_1(t)^2}{x_2(t)^2},\label{eqn:7}\\
\dot{x}_2(t)&=-\bigg(b_2-\frac{[222]}{2d_2}-\frac{[112]}{d_2}\bigg)+\frac{[122]}{d_2}\frac{x_1(t)}{x_2(t)}-\frac{[112]}{2d_2}\frac{x_2(t)^2}{x_1(t)^2},\label{eqn:8}
\end{align}
together with the condition $x_1(t), x_2(t)>0$. 

Let
\begin{align}
A_1&=b_1-\frac{[111]}{2d_1}-\frac{[122]}{d_1},\,\,B_1=\frac{[112]}{d_1},\,\,C_1=\frac{[122]}{2d_1},\label{eqn:427}\\
A_2&=b_2-\frac{[222]}{2d_2}-\frac{[112]}{d_2},\,\,B_2=\frac{[122]}{d_2},\,\,C_2=\frac{[112]}{2d_2}.\label{eqn:428}
\end{align}
\begin{remark}\label{rem:1}
Because of the relations \eqref{eqn:5} and the fact that $K$ is maximal in $G$, the quantities $A_i$, $B_i$ and $C_i$, with $i=1,2$, defined above are strictly positive.
\end{remark}
Consider
\begin{equation*}
y(t)=\frac{x_1(t)}{x_2(t)}.
\end{equation*}
Then, the system \eqref{eqn:7}-\eqref{eqn:8} becomes
\begin{align}
\dot{x}_1(t)&=-A_1+\frac{B_1}{y(t)}-C_1y^2(t),\label{eqn:10}\\
\dot{x}_2(t)&=-A_2+B_2\,y(t)-\frac{C_2}{y(t)^2}.\label{eqn:11}
\end{align}
The evolution equation for $y(t)$ is given by
\begin{align*}
\dot{y}(t)&=\frac{1}{y(t)x_2(t)}(y(t)\dot{x}_1(t)-y(t)^2\dot{x}_2(t))\\
&=\frac{1}{y(t)x_2(t)}\big(-\big(B_2+C_1\big)\,y(t)^3+A_2\,y(t)^2-A_1y(t)+B_1+C_2\big).
\end{align*}
Let
\begin{align*}
g_1(y(t))&=y(t)\dot{x}_1(t)=-C_1y(t)^3-A_1\,y(t)+B_1,\\
g_2(y(t))&=y(t)^2\dot{x}_2(t)=B_2\,y(t)^3-A_2\,y(t)^2-C_2,
\end{align*}
which are two cubics in $y(t)$ with the following properties. 
The cubic $g_1(y(t))$ tends to $-\infty$ when $y(t)\rightarrow+\infty$ and it equals $B_1>0$ when $y(t)=0$. 
Moreover, $g_1(y(t))$ is monotonically decreasing in $y(t)$.
Hence, it only has one root, which is strictly positive.
On the other hand, the cubic $g_2(y(t))$ tends to $+\infty$ when $y(t)\rightarrow+\infty$ and it becomes $-C_2<0$ when $y(t)=0$.
We also have that $g_2(y(t))$ has two critical points: one at $y(t)=0$, which is a local maximum, and one at a positive $y(t)$, which is a local minimum. 
Hence, as $y(t)$ increases, $g_2(y(t))$ decreases until it reaches its minimum and then it increases monotonically in $y(t)$.
We can then conclude that also $g_2(y(t))$ has only one root, which is strictly positive.
As $y(t)>0$ for all $t$, the zeroes of $g_1(y(t))$ and $g_2(y(t))$ correspond exactly to the critical points of $\dot{x}_1(t)$ and $\dot{x}_2(t)$, respectively. 
We denote these points by $\tilde{y}_1$ and $\tilde{y}_2$, respectively.
Then, note that $y(t)$ is a root of the equation
\begin{equation}\label{eqn:9}
-\big(B_2+C_1\big)\,y(t)^3+A_2\,y(t)^2-A_1y(t)+B_1+C_2=0
\end{equation}
if and only if $g_1(y(t))=g_2(y(t))$. 
We have that the roots of \eqref{eqn:9} are always strictly positive.
In fact, if $y(t)$ is negative, by remark \ref{rem:1}, the above expression is strictly positive. 
Moreover, these roots are located between $\tilde{y}_1$ and $\tilde{y}_2$ defined above. 
In particular, this tells us that $\tilde{y}_1<\tilde{y}_2$. 
In fact, if $\tilde{y}_2<\tilde{y}_1$, $\dot{x}_1(t)$ and $\dot{x}_2(t)$ are positive for all $t$ such that $y(t)\in(\tilde{y}_2,\tilde{y}_1)$. 
Now, the fact that the roots of \eqref{eqn:9} are all strictly positive implies that equation \eqref{eqn:9} has solutions if and only if $G/K$ carries homogeneous Einstein metrics. 
As the space is compact, every invariant Einstein metric has positive scalar curvature. 
If we start the Ricci flow at a homogeneous Einstein metric, the whole space shrinks to a point in finite time. 
So, $x_1(t)$ and $x_2(t)$ cannot be increasing when $y(t)$ equals to an invariant Einstein metric. 
Hence $\tilde{y}_1<\tilde{y}_2$ and $\dot{x}_1(t)$ and $\dot{x}_2(t)$ are negative for all $t$ such that $y(t)\in(\tilde{y}_1,\tilde{y}_2)$. 

By the above discussion, $G/K$ carries at most three $G$-invariant Einstein metrics, up to scaling, which correspond to the roots of \eqref{eqn:9}.
In particular, $G/K$ carries at least one $G$-invariant Einstein metric (cf. \cite[Theorem 2.2]{Wang-Ziller:1}). 
In fact, 
\begin{equation*}
-\big(B_2+C_1\big)\,y^3+A_2\,y^2-A_1y+B_1+C_2
\end{equation*}
tends to $-\infty$ when $y\rightarrow+\infty$ and it becomes $B_1+C_2>0$ when $y=0$, which implies that it has to have at least one positive root.
Hence we need to distinguish the following three cases:
\begin{enumerate}
\renewcommand{\theenumi}{(\alph{enumi}}
\setcounter{enumi}{\value{enumi_non_max}}
\item $G/K$ carries three homogeneous Einstein metrics, up to scaling;
\item $G/K$ carries two homogeneous Einstein metrics, up to scaling;
\item $G/K$ carries one homogeneous Einstein metric, up to scaling.
\end{enumerate}
\begin{remark}
Note that $y(t)$ is monotone along any solution to the HRF. 
This is due to the uniqueness of the solution and the fact that the critical points of $y(t)$ correspond to homogeneous Einstein metrics on $G/K$. 
\end{remark}
Before considering the three cases listed above separately, some remarks about the ODE are as follows. 
By standard ODE theory, there exists $T\leq\infty$ such that there exists a unique solution to the HRF on $G/K$ which is defined on the maximal time interval $[0,T)$.
Moreover, the solution exists as long as both functions are positive and the norm of the solution is bounded.
Consider the phase space 
\begin{equation*}
X=\{(x_1,x_2)\in\mathbb R^2,x_1,x_2>0\}, 
\end{equation*}
to which the solution $(x_1(t),x_2(t))$ to the HRF belongs. 
We have that $x_1=\tilde{y}_1x_2$ and $x_1=\tilde{y}_2x_2$ define two lines in $X$, which separate the regions in which $x_1(t)$ and $x_2(t)$ are monotonically increasing or decreasing.
As $\tilde{y}_1<\tilde{y}_2$, $X$ can be divided into three connected regions:
\begin{enumerate}
\item $X_1=\{(x_1,x_2)\in X, x_1<\tilde{y}_1x_2\}$,
\item $X_2=\{(x_1,x_2)\in X, \tilde{y}_1x_2<x_1<\tilde{y}_2x_2\}$,
\item $X_3=\{(x_1,x_2)\in X, x_1>\tilde{y}_2x_2\}$.
\end{enumerate}
We note that the only region which is invariant under the HRF is the one given by $X_2$ above, as the tangent vector $(\dot{x}_1(t),\dot{x}_2(t))$ on boundary
\begin{equation*}
\{(x_1,x_2)\in X, x_1=\tilde{y}_1x_2\}\cup\{(x_1,x_2)\in X, x_1=\tilde{y}_2x_2\}
\end{equation*}
of $X_2$ always points inside it.
Moreover, in this region both $x_1(t)$ and $x_2(t)$ are monotonically decreasing.
On the other hand, $X_1$ is characterised by $x_1(t)$ being monotonically increasing and bounded and $x_2(t)$ being monotonically decreasing.
Whereas, in $X_3$, $x_1(t)$ is monotonically decreasing and $x_2(t)$ is monotonically increasing and bounded.
We then have that if $(x_1(0),x_2(0))$ belongs to $X_1$ or $X_3$ above, the HRF $(x_1(t),x_2(t))$ with this initial condition will always enter the region $X_2$ in finite time and will stay there.
Once in $X_2$, as $x_1(t)$ and $x_2(t)$ are monotonically decreasing in $t$, the solution to the HRF exists and it is unique as long as both $x_1(t)$ and $x_2(t)$ are strictly positive.
By the mean value theorem and \eqref{eqn:7}-\eqref{eqn:8} $x_1(t)$ and $x_2(t)$ can only go to zero simultaneously, as $t\rightarrow T$, and in such a way that $\frac{x_1(t)}{x_2(t)}$ remains bounded and strictly positive. 
Hence, the flow will stop at a finite time $T$, which is characterised by both $x_1(t)$ and $x_2(t)$ becoming zero.

\subsection{Case (d)}
In this case, there are three different solutions to \eqref{eqn:9}. 
We will denote them by $y_1,y_2$ and $y_3$. 
Suppose without loss of generality that $y_1<y_2<y_3$. 
The evolution equation of $y(t)$ is given by
\begin{equation*}
\dot{y}(t)=-\frac{1}{x_2(t)y(t)}\big(B_2+C_1\big)\big(y(t)-y_1\big)\big(y(t)-y_2\big)\big(y(t)-y_3\big).
\end{equation*}
We have four possible initial conditions:
\begin{enumerate}
\renewcommand{\theenumi}{(d)(\arabic{enumi}}
\item $y(0)<y_1$,
\item $y_1<y(0)<y_2$,
\item $y_2<y(0)<y_3$,
\item $y(0)>y_3$.
\end{enumerate}
\begin{remark}
We observe that if we start the Ricci flow at $y(0)=y_i$, then the solution will be given by $y(t)=y_i$, for all $i=1,\dots,4$.
\end{remark}
Because of the above remark and the uniqueness of the solution, we have the following proposition.
\begin{proposition}
The initial conditions \textup{(d)(1)}, \textup{(d)(2)}, \textup{(d)(3)} and \textup{(d)(4)} are preserved under the HRF.
\end{proposition}
We will now analyse these different initial conditions separately. 
We begin by noticing the following. 
If $y(0)$ satisfies \textup{(d)(1)} or \textup{(d)(3)}, then $y(t)$ will be increasing in $t$. 
On the contrary, if $y(0)$ satisfies \textup{(d)(2)} or \textup{(d)(4)}, then $y(t)$ will be decreasing in $t$. 
Moreover, as
\begin{align*}
&\ddot{x}_1(t)=-\frac{B_1}{y(t)^2}\dot{y}(t)-2C_1y(t)\dot{y}(t)=-\bigg(\frac{B_1}{y(t)^2}+2C_1y(t)\bigg)\dot{y}(t),\\
&\ddot{x}_2(t)=B_2\dot{y}(t)+2\frac{C_2}{y(t)^3}\dot{y}(t)=\bigg(B_2+2\frac{C_2}{y(t)^3}\bigg)\dot{y}(t),
\end{align*}
we also have that $\tilde{y}_1$ and $\tilde{y}_2$ are maximum points of $x_1(t)$ and $x_2(t)$, respectively.
By performing an ODE analysis of \eqref{eqn:10}-\eqref{eqn:11} very similar to the one previously done, we can prove the following theorem.
\begin{theorem}
If \textup{(d)(1)} is satisfied, then there exists $T<\infty$ such that there exists a unique solution to the HRF on $G/K$ defined on the maximal time interval $[0,T)$. 
$T$ is a type I singularity and, as $t\rightarrow T$, the whole space shrinks to a point and $y(t)$ approaches the invariant Einstein metric $y_1$. 
In this case, there are no ancient solutions.

If \textup{(d)(2)} is satisfied, then there exists $T<\infty$ such that there exists a unique type I ancient solution to the HRF on $G/K$ which defined on the maximal time interval $(-\infty,T)$. 
$T$ is a type I singularity and, as $t\rightarrow T$, $G/K$ shrinks to a point and $y(t)\rightarrow y_1$. 
Moreover, as $t\rightarrow-\infty$, $y(t)\rightarrow y_2$.

If \textup{(d)(3)} is satisfied, then there exists $T<\infty$ such that there exists a unique type I ancient solution to the HRF on $G/K$ which is defined on the maximal time interval $(-\infty,T)$. 
$T$ is a type I singularity and, as $t\rightarrow T$, $G/K$ shrinks to a point and $y(t)\rightarrow y_3$. 
Furthermore, as $t\rightarrow-\infty$, $y(t)\rightarrow y_2$.

Finally, if \textup{(d)(4)} is satisfied, then there exists $T<\infty$ such that there exists a unique solution to the HRF on $G/K$ defined on the maximal time interval $[0,T)$. 
$T$ is a type I singularity and, as $t\rightarrow T$, $G/K$ shrinks to a point and $y(t)\rightarrow y_3$. 
There are no ancient solutions in this case.
\end{theorem}

\subsection{Case (e)}
We suppose now that \eqref{eqn:9} has 2 distinct roots. 
Let us denote them by $y_1$ and $y_2$. 
In this case, we can write the evolution equation of $y(t)$ in the following way:
\begin{equation*}
\dot{y}(t)=-\frac{1}{x_2(t)y(t)}\big(B_2+C_1\big)\big(y(t)-y_1\big)\big(y(t)-y_2\big)^2.
\end{equation*}
We then need to distinguish two possible situations, i.e. $y_1<y_2$ or $y_1>y_2$. 
Moreover, for each one of them, we have three possible initial conditions. 
If $y_1<y_2$, we can have
\begin{enumerate}
\renewcommand{\theenumi}{(e)(\arabic{enumi}}
\item $y(0)<y_1$,
\item $y_1<y(0)<y_2$,
\item $y(0)>y_2$.
\newcounter{enumi_saved}
\setcounter{enumi_saved}{\value{enumi}}
\end{enumerate}
In this case, $y(t)$ will be increasing in $t$ for $y(0)$ satisfying \textup{(e)(1)} and it will be decreasing in $t$ otherwise. 
If $y_1>y_2$, we can have
\begin{enumerate}
\renewcommand{\theenumi}{(e)(\arabic{enumi}}
\setcounter{enumi}{\value{enumi_saved}}
\item $y(0)<y_2$,
\item $y_2<y(0)<y_1$,
\item $y(0)>y_1$.
\end{enumerate}
Then, $y(t)$ will be increasing in $t$ when $y(0)$ satisfies \textup{(e)(4)} or \textup{(e)(5)} and it will be decreasing in $t$ otherwise. 

Note that, in both situations, because of the uniqueness of the solution to the HRF, these initial conditions are preserved forward and backwards in time. 
By studying the ODEs \eqref{eqn:10}-\eqref{eqn:11} in the various possible situations, we can prove the following theorem.
\begin{theorem}
Suppose that $y_1<y_2$ (resp. $y_1>y_2$). 
Then if $y(0)<y_1$ (resp. $y(0)<y_2$), there exists $T<\infty$ such that there exists a unique solution to the HRF on $G/K$ which is defined on the maximal time interval $[0,T)$. 
$T$ is a type I singularity and, as $t\rightarrow T$, $G/K$ shrinks to a point and $y(t)\rightarrow y_1$ (resp. $y(t)\rightarrow y_2$). 
There are no ancient solutions in this case.

If $y_1<y(0)<y_2$ (resp. $y_2<y(0)<y_1$), there exist a positive constant $T<\infty$ such that there exists a unique type I ancient solution on $G/K$ which is defined on the maximal time interval $(-\infty,T)$. $T$ is a type I singularity and, as $T\rightarrow T$, $G/K$ shrinks to a point. Moreover, the HRF flows $y(t)$ from $y_2$ (resp. $y_1$) to $y_1$ (resp. $y_2$), as $t$ goes from $-\infty$ to $T$.

Finally, if $y(0)>y_2$ (resp. $y(0)>y_1$), there exists $T<\infty$ such that there exists a unique solution to the HRF on $G/K$ which is defined on the maximal time interval $[0,T)$. 
$T$ is a type singularity and, as $t\rightarrow T$, $G/K$ shrinks to a point and $y(t)\rightarrow y_2$ (resp. $y(t)\rightarrow y_1$). 
In particular, there are no ancient solutions.
\end{theorem}

\subsection{Case (f)}
In this case, equation \eqref{eqn:9} has exactly one root, which will be denoted by $\bar{y}$. 
We have two possible situations. 
Either $\bar{y}$ has order three or it has order one. 
In the first case, the evolution equation for $y(t)$ is given by
\begin{equation*}
\dot{y}=-\frac{1}{x_2(t)y(t)}\big(B_2+C_1\big)\big(y(t)-\bar{y}\big)^3.
\end{equation*}
In the second case, the evolution equation for $y(t)$ becomes
\begin{equation*}
\dot{y}=-\frac{1}{x_2(t)y(t)}\big(B_2+C_1\big)\big(y(t)-\bar{y}\big)P(y(t)),
\end{equation*}
where $P(y(t))$ is a polynomial of degree two in $y(t)$ which is strictly positive for all $t$. 
In both cases, if $y(0)<\bar{y}$, $y(t)$ will be increasing in $t$, while, if $y(0)>\bar{y}$, $y(t)$ will be decreasing in $t$. 
Note that because of the uniqueness of the solution to the HRF, these initial conditions are preserved. 
Then, the analysis of the ODE system \eqref{eqn:10}-\eqref{eqn:11} leads to the following theorem.
\begin{theorem}
There exists $T<\infty$ such that there exists a unique solution to the HRF on $G/K$ which is defined on the maximal time interval $[0,T)$. 
$T$ is a type singularity and, as $t\rightarrow T$, $G/K$ shrinks to a point and $y(t)\rightarrow\bar{y}$. 
In particular, there are no ancient solutions.
\end{theorem}

\subsection{The scalar curvature}
In this case, the scalar curvature of $g(t)$ is given by
\begin{equation*}
R(t)=\frac{A_1d_1}{2}\frac{1}{x_1(t)}+\frac{d_2A_2}{2}\frac{1}{x_2(t)}-\frac{d_1}{4}B_1\frac{x_2(t)}{x_1(t)^2}-\frac{d_2}{4}B_2\frac{x_1(t)}{x_2(t)^2}.
\end{equation*}
Let us begin by finding the regions in which the scalar curvature has a definite sign.
We have that $R(t)=0$ if and only if
\begin{equation*}
\frac{d_1}{2}A_1+\frac{d_2}{2}A_2\frac{x_1(t)}{x_2(t)}-\frac{d_1}{4}B_1\frac{x_2(t)}{x_1(t)}-\frac{d_2}{4}B_2\frac{x_1(t)^2}{x_2(t)^2}=0.
\end{equation*}
Let $y(t)=\frac{x_1(t)}{x_2(t)}$.
Then $R(t)=0$ if and only if $y(t)$ solves
\begin{equation}\label{eqn:423}
-\frac{d_2}{4}B_2y(t)^3+\frac{d_2}{2}A_2y(t)^2+\frac{d_1}{2}A_1y(t)-\frac{d_1}{4}B_1=0.
\end{equation}
Now, as the cubic
\begin{equation*}
\alpha(y):=-\frac{d_2}{4}B_2y^3+\frac{d_2}{2}A_2y^2+\frac{d_1}{2}A_1y-\frac{d_1}{4}B_1
\end{equation*}
intersects the line $y=0$ in a negative value and it tends to $+\infty$, as $y\rightarrow-\infty$, and to $-\infty$, as $y\rightarrow+\infty$, we have that it has to intersect the $y$-axis in at least one negative point.
By \cite[Theorem 2.2]{Wang-Ziller:1}, we know that there has to exist a region in the first quadrant in which the scalar curvature is positive, because the homogneous space $G/K$ admits at least one invariant Einstein metric.
Moreover, by \cite[Theorem 2.1]{Wang-Ziller:1}, we know that the scalar curvature cannot be bounded from below.
Hence, there has to exists a region in the first quadrant in which the scalar curvature is negative.
This implies that \eqref{eqn:423} has at least one positive root.
So the cubic $\alpha(y)$ meets the $y$-axis in at least two points, one positive and one negative.
We also have that $\alpha(y)=0$ if and only if
\begin{equation*}
-\frac{d_2}{4}B_2y^3=-\frac{d_2}{2}A_2y^2-\frac{d_1}{2}A_1y+\frac{d_1}{4}B_1.
\end{equation*}
Let
\begin{align*}
\alpha_1(y)&=-\frac{d_2}{4}B_2y^3\\
\alpha_2(y)&=-\frac{d_2}{2}A_2y^2-\frac{d_1}{2}A_1y+\frac{d_1}{4}B_1.
\end{align*}
Observe that $\alpha_1(0)=0$, while $\alpha_2(0)=\frac{d_1}{4}B_1>0$.
Moreover, both these two curves tend to $-\infty$ and in such a way that $\alpha_1(y)<\alpha_2(y)$, as $y\rightarrow+\infty$.
As they intersect in one positive $y$ and $\alpha_1(y)<\alpha_2(y)$, as $y\rightarrow+\infty$, they have to intersect another time for $y>0$.
So they intersect into two positive points.
Then, $\alpha(y)=0$ has three roots.
Two of these roots have to be positive and the other one has to be negative.
We have then proved that $R(t)=0$ has three roots in terms of $y(t)$, two of them being positive and one being negative.
Let $\bar{y}_1$, $\bar{y}_2$ and $\bar{y}_3$ denote these three roots.
Suppose without loss of generality that $\bar{y}_1$ and $\bar{y}_2$ are positive and that $\bar{y}_3$ is negative.
We can then write $R(t)$ as
\begin{equation*}
\begin{aligned}
R(t)=&\frac{1}{x_1(t)y(t)}\bigg(\frac{d_1}{2}A_1y(t)+\frac{d_2}{2}A_2y(t)^2-\frac{d_1}{4}B_1-\frac{d_2}{4}B_2y(t)^3\bigg)\\
=&\frac{x_2(t)}{x_1(t)^2}\bigg(-\frac{d_2}{4}B_2\bigg)(y(t)-\bar{y}_1)(y(t)-\bar{y}_2)(y(t)-\bar{y}_3).
\end{aligned}
\end{equation*}
$\bar{y}_i$, with $i=1,2,3$ defines lines through the origin in $\mathbb R^2$, which separates the regions in which the scalar curvature is either positive or negative.
We have that the following lemma is true.
\begin{lemma}
If $x_1(t),x_2(t)>0$, then $R(t)$ is positive in the region where both functions are decreasing in $t$.
\end{lemma}
\begin{proof}[Proof of the lemma]
From the evolution equation \eqref{eqn:7}-\eqref{eqn:8}, we have that $x_1(t)$ and $x_2(t)$ are both decreasing in $t$ if and only if
\begin{align*}
-A_1+\frac{B_1}{y(t)}-C_1y(t)^2&<0,\\
-A_2+B_2y(t)-\frac{C_2}{y(t)^2}&<0.
\end{align*}
We can rewrite the above inequalities as follows.
\begin{align*}
A_1y(t)-B_1+C_1y(t)^3&>0,\\
A_2y(t)^2-B_2y(t)^3+C_2&>0.
\end{align*}
Then, using \eqref{eqn:427}-\eqref{eqn:428}, we have that
\begin{align*}
d_1B_1&=2d_2C_2,\\
d_2B_2&=2d_1C_1,
\end{align*}
which implies that $x_1(t)$ and $x_2(t)$ are both decreasing in $t$ if and only if $y(t)$ satisfies the system
\begin{align*}
\frac{d_1}{2}A_1y(t)-\frac{d_1}{2}B_1+\frac{d_2}{4}B_2y(t)^3&>0,\\
\frac{d_2}{2}A_2y(t)^2-\frac{d_2}{2}B_2y(t)^3+\frac{d_1}{4}B_1&>0.
\end{align*}
By putting these two inequalities together, we have that $y(t)$ satisfies the following inequality:
\begin{equation*}
\frac{d_1}{2}A_1y(t)-\frac{d_1}{4}B_1+\frac{d_2}{2}A_2y(t)^2-\frac{d_2}{4}B_2y(t)^3>0,
\end{equation*}
which implies that $R(t)$ has to be positive.
\end{proof}
This lemma implies that the line $x_1=\bar{y}_1x_2$ is located in the region where $x_1(t)$ is increasing and $x_2(t)$ is decreasing, while the line $x_1=\bar{y}_2x_2$ is located in the region where $x_1(t)$ is decreasing and $x_2(t)$ is increasing.
We also have that the following corollary is true.
\begin{corollary}
If $x_1(t),x_2(t)<0$, then $R(t)$ is negative in the region where both functions are decreasing in $t$.
\end{corollary}
The following picture illustrates the regions in which $R(t)$ is positive and those in which it is negative.
\begin{center}
\begin{tikzpicture}[>=stealth,scale=.4]
\draw [very thin,->] (-6,0) -- (6,0) node [fill=white,below=1pt] {$\scriptstyle x_2$} coordinate (x axis);
\draw [very thin,->] (0,-6) -- (0,6) node [fill=white,right=1pt] {$\scriptstyle x_1$} coordinate (y axis);
\draw (-3,-6) -- (3,6) (-6,2.5) -- (6,-2.5) (-6,-3)--(6,3);
\draw (3,6.2) node [right] {$\scriptstyle y=\bar y_1$} (6,3.2) node [right] {$\scriptstyle y=\bar y_2$} (-6,2.8) node [left] {$\scriptstyle y=\bar y_3$};
\draw (1,4) node [above] {$\scriptstyle R<0$} (-1,-4) node [below] {$\scriptstyle R>0$} (-2,-3) node [left] {$\scriptstyle R<0$};
\draw (-4,-1.5) node [left] {$\scriptstyle R>0$} (4,3) node [left] {$\scriptstyle R>0$} (5,1.5) node [right] {$\scriptstyle R<0$};
\draw (-4,0.5) node[left] {$\scriptstyle R>0$} (-3,4.5) node [left] {$\scriptstyle R<0$};
\draw (1,-3.5) node [right] {$\scriptstyle R>0$} (5,-0.7) node [left] {$\scriptstyle R<0$};
\end{tikzpicture}
\end{center}
We now have to analyse the behaviour of the scalar curvature in the first quadrant, where $x_1(t)$ and $x_2(t)$ are both positive.
In this case, we have that $R(t)>0$ if and only if $\bar{y}_1<y(t)<\bar{y}_2$, which includes in particular the region in which both functions are decreasing in $t$.
As the Ricci flow preserves positivity of the scalar curvature, we know that if $R(0)>0$ then $R(t)>0$, for every $t$ such that a solution exists.
If the solution has negative scalar curvature initially, then it means that the initial condition does not lie in the invariant region.
As the solution has to enter the invariant region in finite time, this means that the scalar curvature has to turn positive in finite time, if it was negative initially.

\section{Blowing up the solution near the singularity}\label{sec:4}
We are now going to analyse the rescaled Ricci flow near the singularity $T$, where the singular behaviours are the ones described in theorem \ref{theo:7}.

Let $T$ be a type I singularity for the HRF on $G/K$.
Consider a sequence $\{(p_j,t_j)\}_{j=1}^{\infty}$, with $p_j\in G/K$ and $t_j\rightarrow T$, such that
\begin{equation*}
|\Rm|(p_j,t_j)=\sup_{p\in G/K,\,t\in[0,t_j]}{|\Rm(g(t))|_{g(t)}(p,t)}\rightarrow+\infty.
\end{equation*}
Let
\begin{equation*}
x_k^j(t)=|\Rm|(p_j,t_j)\,x_k\Bigg(t_j+\frac{t}{|\Rm|(p_j,t_j)}\Bigg),
\end{equation*}
for all $k=1,2$. Let $g^j(t)$ be the Riemannian metric defined by $(x_1^j(t),x_2^j(t))$.
Then,
\begin{equation*}
(N,g_\infty(t),p_\infty)=\lim_{j\rightarrow\infty}{(G/K,g^j(t),p_j)}
\end{equation*}
is an ancient Ricci flow. 
Note that the limit of the pointed convergence is not affected by the location of the $p_j$'s, because we are in the homogeneous case. 
By \cite{Enders-Muller-Topping:1}, we know that $(N, g_\infty(t))$ is a nonflat gradient shrinking Ricci soliton. 

In the case where the whole space shrinks to a point in finite time, it is easy to see that $g_\infty(t)$ is given by a homogeneous Einstein metric (with positive scalar curvature) on $G/K$.

Now, let us consider the other kind of singular behaviour which leads to a type I singularity in the HRF. 
Let $g_\infty(t)$ be given by $(x_1^\infty(t),x_2^\infty(t))$. 
By the proof of theorem \ref{theo:7}, we have that
\begin{equation*}
x_1^\infty(t)=\lim_{j\to\infty}{|\Rm|(p_j,t_j)x_1^j\Bigg(t_j+\frac{t}{|\Rm|(p_j,t_j)}\Bigg)}
\end{equation*}
will be given by a decreasing linear function of $t$. 
On the other hand, we have that
\begin{equation*}
x^\infty_2(t)=\infty,
\end{equation*}
as $x_2(t)$ remains bounded, as $t\rightarrow T$.

Hence, $(N,g_\infty(t))$ will be given by the \emph{rigid Ricci soliton}
\begin{equation*}
\widetilde{N}\times\mathbb R^q,
\end{equation*}
where $\widetilde{N}$ is the homogeneous Einstein manifold defined by $x_1^\infty(t)$, and it corresponds to the fibre $H/K$ in \eqref{eqn:23}, and the flat factor $\mathbb R^q$ is defined by $x_2^\infty(t)$, with $q=\dim(G/H)$.

\section{Remarks on the pseudo-Riemannian case}\label{sec:7}
In \cite{Buzano:2}, the author also considered the situation in which the functions $x_1(t)$ and $x_2(t)$ are not necessarily strictly positive.
Geometrically, this means that we allow the initial metric to be pseudo-Riemannian.
We will now describe briefly the main differences with the Riemannian case.

\subsection{When the isotropy group is maximal}
Let $x_1$ and $x_2$ be coordinates in the phase space with $x_1$ being the vertical axis and $x_2$ the horizontal one.
When both $x_1(t)$ and $x_2(t)$ are negative, the monotone quantity $y(t)$ is used to find immortal solutions.
These solutions are defined on $(-T,+\infty)$, with $0<T<\infty$.
We find that when there exists an immortal solution, as $t\rightarrow+\infty$, the flow always converges to an invariant Einstein metric.
In particular, if there are no invariant Einstein metrics, then there are no immortal solutions.

When $x_1(t)>0$, $x_2(t)<0$ and $C=0$, we prove that every initial condition leads to an immortal solution with curvature which becomes unbounded as $t\rightarrow+\infty$.
Moreover, as $t\rightarrow+\infty$, $(G/K,g(t))$ converges in the Gromov-Hausdorf sense to $G/H$ equipped with an invariant Einstein metric.
If $C>0$, we find that every solution arising in the second quadrant stops in finite time because the metric collapses in the direction of the fibre $H/K$.
In this case, the scalar curvature becomes positive in finite time.

Finally, when $x_1(t)<0$ and $x_2(t)>0$ and $C>0$, every solution comes from one arising in the first quadrant.
Moreover, every solution stops in finite time because $x_2(t)$ becomes zero and the scalar curvature has to turn positive in finite time, if it was negative initially.
If $C=0$, then every initial condition leads to an ancient solution, which converges in the Gromov-Hausdorff sense to $G/H$ with a $G$-invariant Einstein metric, as $t\rightarrow-\infty$.

\subsection{When the isotropy group is not maximal}
In the case where $x_1(t),x_2(t)<0$, we find that for certain classes of initial conditions, there exist immortal solutions to the flow, which always converge to an invariant Einstein metric, as $t\rightarrow+\infty$.
This is done once again by using the monotone quantity $y(t)$.

In the remaining two cases, we find that every solution stops in finite time, because the metric collapses in one direction.
Moreover, for these solutions the scalar curvature always has to turn positive in finite time, if it was negative initially.

\bigskip

We conclude by mentioning that examples of homogenous spaces to which the results of this paper can be applied to can be found in \cite{Wang-Ziller:1}, \cite{Dickinson-Kerr:1} and \cite[Appendix A]{He:1}.

\bibliographystyle{amsplain}
\bibliography{/Users/mariabuzano/Documents/Texinput/citazioni}

\end{document}